\documentclass[12pt]{amsart}

\usepackage[T1]{fontenc}
\usepackage[utf8]{inputenc}

\usepackage{microtype}

\usepackage{amsmath,amssymb,amsthm,amsfonts}

\usepackage{thmtools}
\declaretheorem[name=Theorem]{thm}
\declaretheorem[name=Proposition]{prop}

\declaretheorem[name=Lemma, style=definition]{lemma}
\declaretheorem[name=Corollary, style=definition]{cor}

\declaretheorem[name=Remark, numbered=no, style=remark]{remark}

\usepackage{slashed}
\usepackage{dsfont}
\usepackage{mathrsfs}

\usepackage{tikz-cd}
\usepackage{enumitem}

\usepackage{hyperref}
\usepackage{mathtools}
\mathtoolsset{showonlyrefs}

\usepackage{tikz-cd}
\usepackage{todonotes}

\tikzset{
	pf/.style={commutative diagrams/.cd, every arrow, every label},
	surj/.style=commutative diagrams/two heads,
	inj/.style=commutative diagrams/hook,
	gl/.style=commutative diagrams/equal,
	mat/.style={matrix of math nodes, commutative diagrams/.cd, every cell},
	dr/.style={matrix of math nodes, commutative diagrams/.cd, every cell, column sep=small},
	seq/.style={matrix of math nodes, commutative diagrams/.cd, every cell,column sep=small}
}

\newenvironment{diag*}
	{\[\begin{tikzpicture}[commutative diagrams/.cd, every diagram]}
	{\end{tikzpicture}\]\ignorespacesafterend}

\newenvironment{diag}
	{\begin{equation}\begin{tikzpicture}[commutative diagrams/.cd, every diagram]}
	{\end{tikzpicture}\end{equation}\ignorespacesafterend}

\usepackage[margin=0.9in]{geometry}

\DeclareMathOperator{\Rm}{R}
\DeclareMathOperator{\diverg}{div}
\DeclareMathOperator{\tr}{Tr}
\DeclareMathOperator{\Diff}{Diff}
\DeclareMathOperator{\Sym}{Sym}
\DeclareMathOperator{\Aut}{Aut}
\DeclareMathOperator{\Spin}{Spin}
\DeclareMathOperator{\SO}{SO}
\DeclareMathOperator{\grad}{grad}
\DeclareMathOperator{\End}{End}
\DeclareMathOperator{\id}{Id}
\DeclareMathOperator{\Cl}{Cl}
\DeclareMathOperator{\Ima}{Im}

\newcommand{\ACM}{J_M} 
\newcommand{\SpinPFB}{P_{\Spin}(M,g)} 
\newcommand{\SpinorBundleM}{S}  
\newcommand{\SpinorBundleP}{\Sigma} 
\newcommand{\ACSBP}{J_\SpinorBundleP} 
\newcommand{\CliffordBundleM}{{\operatorname{Cl}\left(M, -g\right)}}

\newcommand{\dd}{\mathop{}\!\mathrm{d}}
\newcommand{\A}{\mathbb{A}}
\newcommand{\D}{\slashed{D}}
\newcommand{\p}{\partial}
\newcommand{\al}{\alpha}
\newcommand{\be}{\beta}
\newcommand{\na}{\nabla}
\newcommand{\R}{\mathbb{R}}
\newcommand{\C}{\mathbb{C}}
\newcommand{\pd}{\slashed{\partial}}

\newcommand{\map}{\phi}
\newcommand{\dv}{\dd{vol}}
\newcommand{\diff}{f}
\newcommand{\Lie}{\mathcal{L}}
\newcommand{\dt}[1]{\frac{\dd}{\dd{t}}\Big|_{t=#1}}

\newcommand{\ov}[1]{\overline{#1}}

\newcommand{\gs}{g_s}
\newcommand{\Met}{\mathfrak{Met}}
\newcommand{\gco}{g^\vee}

\title[Symmetries and conservation laws]{Symmetries and conservation laws of\\ a nonlinear sigma model with gravitino}

\begin{document}

\author{Jürgen Jost, Enno Keßler, Jürgen Tolksdorf, Ruijun Wu, Miaomiao Zhu}

\address{Max Planck Institute for Mathematics in the Sciences\\Inselstr. 22--26\\D-04103 Leipzig, Germany}
	\email{jjost@mis.mpg.de}

\address{Max Planck Institute for Mathematics in the Sciences\\Inselstr. 22--26\\D-04103 Leipzig, Germany}
	\email{Enno.Kessler@mis.mpg.de}

\address{Max Planck Institute for Mathematics in the Sciences\\Inselstr. 22--26\\D-04103 Leipzig, Germany}
	\email{Juergen.Tolksdorf@mis.mpg.de}

\address{Max Planck Institute for Mathematics in the Sciences\\Inselstr. 22--26\\D-04103 Leipzig, Germany}
	\email{Ruijun.Wu@mis.mpg.de}

\address{School of Mathematical Sciences, Shanghai Jiao Tong University\\Dongchuan Road 800\\200240 Shanghai, P.R.China}
	\email{mizhu@sjtu.edu.cn}

\thanks{Miaomiao Zhu was supported in part by the National Science Foundation of China (No. 11601325).}

\date{\today}

\begin{abstract}
	We show that the action functional of the nonlinear sigma model with gravitino considered in~\cite{jost2016regularity} is invariant under rescaled conformal transformations, super Weyl transformations and diffeomorphisms.
	We give a careful geometric explanation how a variation of the metric leads to the corresponding variation of the spinors.
	In particular cases and despite using only commutative variables, the functional possesses a degenerate super symmetry.
	The corresponding conservation laws lead to a geometric interpretation of the energy-momentum tensor and supercurrent as holomorphic sections of appropriate bundles.
\end{abstract}

\keywords{nonlinear sigma model, gravitino, Noether's theorem, symmetry, energy-momentum tensor, supercurrent, Dirac-harmonic map}

\maketitle


\section{Introduction}

The main motivation for the introduction of the two-dimensional nonlinear supersymmetric sigma model in quantum field theory, or more specifically super gravity and super string theory, are its symmetries, see for instance~\cite{brink1976locally, deser1976complete}.
Furthermore, as argued in~\cite{kessler2016functional}, the functional is determined by its symmetries together with suitable bounds on the order of its Euler--Lagrange equations.
While super symmetric models are usually formulated using anti-commutative variables, in~\cite{jost2016regularity} an analogue of the two-dimensional nonlinear supersymmetric sigma model using only commutative variables was introduced.
Here we would like to give a detailed geometric account of the symmetries of the purely commutative model considered in~\cite{jost2016regularity}.

In physics as well as in geometry, symmetries of a system always have important implications.
In particular, differentiable symmetries of a system give rise to conservation laws---this is basically what the famous Noether's theorem states, see e.g.~\cite[Chapter 1]{jost1999calculus}, \cite[Section 2.3.2]{jost2009geometry} and~\cite{gelfand1963calculus}.
To see how Noether's principle helps, let us take a closer look at the harmonic map theory as an example; this will also be the prototype for our work.

The \emph{Dirichlet energy} for a smooth map $\map\colon (M,g)\to (N,h)$ from a Riemannian surface to a  Riemannian manifold is given by
\begin{equation}
	E(\map)=\frac{1}{2}\int_M |\dd\map|_{\gco\otimes\map^*h}^2\dv_g
	=\frac{1}{2} \int_M g^{\al\be}\frac{\p\map^i}{\p x^\al}\frac{\p \map^j}{\p x^\be}h_{ij}(\map(x))\sqrt{\det{g(x)}}\dd x
\end{equation}
where $\gco$ denotes the dual metric on the cotangent bundle and $\map^*h$ the induced metric on the pullback bundle $\map^*TN$.
The critical points of the energy functionals are named \emph{harmonic maps} in~\cite{eells1964harmonic}, the study of which has been among the central topics in mathematics for a long time.

Essential in the study of harmonic maps are the diffeomorphism invariance
\footnote{We denote by $g_\diff$ the non-negative definite symmetric 2-form defined by \begin{equation} g_\diff(X,Y)(x)\coloneqq g_{\diff(x)}(T\diff(X),T\diff(Y)) \end{equation} for any $x\in M$ and any $X,Y\in\Gamma(TM)$, which is commonly referred to as ``pullback metric'' when $f$ is an immersion in differential geometry.
	We reserve the symbol \(\diff^*g\) for the induced metric on the pullback bundle \(\diff^*TM\).}
of the Dirichlet energy
\begin{equation}
	E(\map\circ\diff; g_\diff)=E(\map;g), \qquad \forall \diff\in \Diff(M),
\end{equation}
and as a peculiarity in the case that the domain is two dimensional, the conformal invariance
\begin{equation}
	E(\map; e^{2u}g) = E(\map; g), \qquad \forall u\in C^\infty(M).
\end{equation}
Those two symmetries link the theory of harmonic maps on a Riemann surface to the Teichmüller theory, see~\cite{jost2013compact}.

The conservation laws corresponding to diffeomorphism invariance and conformal invariance are expressed in terms of the \emph{energy-momentum tensor}.
The energy-momentum tensor is defined as the variation of the Dirichlet energy with respect to the metric, that is, for a smooth family $(g_t)$ of Riemannian metrics on $M$,
\begin{equation}
	\begin{split}
		\frac{\dd}{\dd t} 2E(\map,g_t)
		&=-\frac{1}{2}\int_M \left\langle\frac{\p g_t}{\p t}, T(\map;g_t)\right\rangle \dv_{g_t}.
	\end{split}
\end{equation}
Explicitly the energy-momentum tensor is a symmetric 2-tensor $T(\map;g_t)=T_{\al\be}\dd x^\al\otimes \dd x^\be$ given by
\begin{equation}
		\frac{1}{2}T_{\al\be} =h_{ij}(\map)\frac{\p\map^i}{\p x^\al}\frac{\p\map^j}{\p x^\be}-\frac{1}{2}|\dd\map|^2_{\gco_t\otimes\map^*h} g_{t\al\be}.
\end{equation}
The energy-momentum tensor is symmetric and traceless---though this is clear from the expression, it roots in the fact that we are taking variations in the symmetric $2$-tensors and that the energy functional is conformally invariant in dimension two.
Actually, consider the conformal family $(g_t=e^{2tu}g)$, the conformal invariance says that $E(\map; e^{2tu}g)=E(\map;g)$, and consequently
\begin{equation}
	\begin{split}
		0&=\dt{0} 2E(\map, e^{2tu}g)
		=-\frac{1}{2}\int_M \left\langle 2u g,T(\map;g)\right\rangle\dv_g\\
		&=-\int_M u\cdot\tr_g\left(T(\map;g)\right) \dv_g,
	\end{split}
\end{equation}
which is equivalent to say that $\tr_g(T(\map;g))=0$ since the function $u$ can be arbitrary.

The conservation law corresponding to the diffeomorphism invariance of the Dirichlet energy is that the energy-momentum tensor is divergence-free on shell
\footnote{Here we use the phrase “on shell”, as it is common in physics, to say that the Euler--Lagrange equations are satisfied.}.
Indeed, for a differentiable family $(f_t)$ of diffeomorphisms of $M$ with $f_0=\id_M$ and a harmonic map \(\map\), we have
\begin{equation}
		0=\dt{0} E(\map\circ\diff_t;g_{\diff_t}) = -\int_M\frac{1}{2} \left\langle\frac{\p}{\p t} g_{\diff_t} \Big|_{t=0}, T(\map;g)\right\rangle \dv_{g},
\end{equation}
When we denote the generator of \(\diff_t\) by $X\in \Gamma(TM)$ the differential of the metric is given by the following Lie-derivative
\begin{equation}
	\dt{0} g_{\diff_t}=\Lie_X g.
\end{equation}
Consequently, when the map $\map$ is harmonic, we have
\begin{equation}
	0=\int_M \langle \Lie_X g, T(\map;g) \rangle \dv_g =-2\int_M \langle X, \diverg_g (T(\map;g))\rangle\dv_g,
\end{equation}
where the second equality is recalled in Section~\ref{sect:formal divergence}.
As the vector field $X$ could be arbitrary, it follows that the energy-momentum tensor is divergence-free.

The fact that the energy-momentum tensor is symmetric, traceless and divergence-free has important geometric and analytic consequences.
For instance, the space of symmetric, traceless and divergence-free 2-tensors can be identified with the space of holomorphic quadratic differentials.
Using existence results of harmonic maps and the Theorem of Riemann--Roch, one can then show that the Teichmüller space is a ball of dimension \(6p-6\), where \(p\) is the genus of the surface.

The action functional in~\cite{jost2016regularity} extends the Dirichlet energy to include spinor fields: a Dirac-term in the field \(\psi\) and additional terms inolving the gravitino field \(\chi\).
While the Euler--Lagrange equations for \(\psi\) yield a Dirac-type equation, the gravitino is more considered as a parameter in the theory, similar to the metric \(g\).
By analogy to the models studied in physics we expect conformal invariance, super Weyl invariance (an analogue to conformal invariance affecting the gravitino) and diffeomorphism invariance.
We will show that the functional introduced in~\cite{jost2016regularity} possesses indeed the aforementioned symmetries.
The major difficulty in the study of those symmetries is that spinors depend on the metric.
In order to express the variation of the spinors in terms of the variation of the metric we use the methods developped in~\cite{bourguignon1992spineurs}.
In contrast to the original physical model, the functional studied here possesses only a “degenerate supersymmetry” in special cases.

The conservation laws corresponding to those four symmetries are expressed in terms of the energy-momentum tensor and the supercurrent, which is the variation of the action with respect to the gravitino.
We show that rescaled conformal and super Weyl invariance lead to algebraic constraints on the energy-momentum tensor \(T\) and the supercurrent \(J\).
Diffeomorphism invariance and degenerate supersymmetry yield equations for the divergences of \(T\) and \(J\).
All four symmetries together allow for an interpretation of \(T\) and \(J\) as holomorphic sections of appropriate bundles.
This is in full analogy to the super geometric setting, where \(T\) and \(J\) constitute tangent vectors to the moduli space of so-called super Riemann surfaces, see~\cite{jost2014super, kessler2016thesis}.

The article is organized as follows:
First we consider a Dirac action as a toy model, as well as to set up the notation for later use.
The Weyl symmetry and the diffeomorphism invariance are checked and the corresponding conservation laws are obtained.
Then we put the main effort to study the nonlinear sigma model with gravitino in~\cite{jost2016regularity}, and the four kinds of symmetries mentioned above are discussed.
We give explicit computation of the energy-momentum tensor and supercurrent in that case.
The corresponding conservation laws are derived and explained.
Finally we make a supplement about the divergence operators, both on symmetric 2-tensors and on spinors, which are frequently used in the context.

\section{Dirac action and its conservation laws}
In this section we consider a Dirac action.
While various Dirac actions play a prominent role in physics, we will use the Dirac action as an elaborate toy model.
At the example of the Dirac action we will set up the notation and theory necessary to study symmetries and conservation laws of a sigma-model involving spinors.

The geometric setup is as in~\cite{jost2016regularity}.
Let $(M,g)$ be a closed oriented Riemannian surface and let $\xi\colon P_{\Spin}(M,g)\to P_{\SO}(M,g)$ be a spin structure on $(M,g)$.
The irreducible representation of the Clifford bundle $\Cl_{0,2}$ induces the real spin representation $\mu\colon \Spin(2)\to GL(V)$ where $V$ is a representation space and is isomorphic to $\R^4$.
From this one can form the associated spinor bundle
\begin{equation}
	S_g\coloneqq P_{\Spin}(M,g)\times_\mu V.
\end{equation}
Note that here by the lower index in $S_g$ we emphasize the dependence on the Riemannian metric~$g$.
There are a canonical fiberwise real inner product denoted by $\gs$, and a canonical metric connection $\na^s$.
The connections on $\Gamma(S_g)$ and on the algebra bundle $\Cl(M,-g)$ satisfy a Leibniz rule, see e.g.~\cite{lawson1989spin}.
Choosing an isomorphism between $V$ and $\Cl_{0,2}$ we get a Clifford map denoted by $\gamma\colon TM\to \End(S_g)$ which satisfies the Clifford relation
\begin{equation}
	\gamma(X)\gamma(Y)+\gamma(Y)\gamma(X)=-2g(X,Y) \qquad \textnormal{ for all } X, Y \in \Gamma(TM).
\end{equation}
Sections of $S_g$ will be referred to as (pure) spinors, which is used to describe matter fields with non-integer spins in physics.

On the spinor bundle $S_g$ there is a spin Dirac operator $\pd_g\colon \Gamma(S_g)\to \Gamma(S_g)$ which is defined by
\begin{equation}
	\pd_g \sigma \coloneqq \gamma(e_\al)\na^s_{e_\al}\sigma \qquad \textnormal{ for  } \sigma\in \Gamma(S_g),
\end{equation}
where $(e_\al)$ is a local oriented $g$-orthonormal frame.
It is easily to check that $\pd_g$ is well-defined, that is, independent of the \(g\)-orthonormal frame.
This Dirac operator is an elliptic self-adjoint operator, and has a real spectrum which is unbounded in either side in $\R$.
A spinor $q\in\Gamma(S)$ is called harmonic if it is in the kernel of $\pd_g$, i.e.~$\pd_g q=0$.
One can refer to~\cite{hitchin1974harmonic, bar1998harmonic} for more knowledge about harmonic spinors.

Here we consider a \emph{Dirac action}
\footnote{Note that in most applications in physics a mass term is added to the Lagrangian.}
as a functional defined on spinors and Riemannian metrics by
\begin{equation}
	DA(\sigma;g)\coloneqq \int_M \gs(\sigma,\pd_g\sigma)\dv_g \qquad \textnormal{ for } g\in \Met(M) \textnormal{ and for }  \sigma\in \Gamma(S_g)
\end{equation}
where $\Met(M)$ denotes the space of all Riemannian metrics on $M$.
From the facts above on the spectrum about $\pd_g$ one knows that this action functional cannot be bounded from either side.
Recall that the spin Dirac operator $\pd_g$ is self-adjoint with respect to the $L^2$-inner product.
An easy calculation shows that for a smooth family ${(\sigma_t)}_t$ of spinors,
\begin{equation}
	\dt{0} DA(\sigma_t;g)=2\int_M \left\langle \left(\dt{0}\sigma_t\right), \pd_g \sigma \right\rangle \dv_g.
\end{equation}
Thus the Euler--Lagrange equation for the Dirac action is $\pd_g \sigma=0$, that is, the critical spinors are the harmonic spinors.

\subsection{Rescaled conformal invariance of the Dirac action}%
\label{sect:ConformalInvarianceDiracAction}
First observe that mathematically rigorously, one can not keep the spinor $\sigma$ fixed while varying the metric, because the spin structure depends on the metric $g$ and so do the spinor bundle $S_g$ and the Dirac operator $\pd_g$.
To overcome this difficulty we take the approach developed in~\cite{bourguignon1992spineurs}.

Let \(g\) and \(g'\) be two Riemannian metrics on \(M\).
There is a unique self-adjoint endomorphism \(H\in\End(TM)\) such that \(g'(\cdot, \cdot) = g(H\cdot, \cdot)\).
Thus if we set
\begin{equation}
	b\equiv b^g_{g'} \coloneqq H^{-1/2} \in \Aut(TM),
\end{equation}
then $b\colon (TM,g)\to (TM, g')$ is an isometry of Riemannian vector bundles.
More explicitly, $ g(\cdot,\cdot)=g'\left(b(\cdot),b(\cdot)\right)$.

Next we lift the isomorphism $b$ to the spin level.
Since $b\colon (TM,g)\to (TM,g')$ is $\SO(2)$-equivariant, it defines a principal bundle isomorphism
\begin{equation}
	b\colon P_{\SO}(M,g)\to P_{\SO}(M,g')
\end{equation}
which lifts to an isomorphism $\tilde{b}\colon P_{\Spin}(M,g)\to P_{\Spin}(M,g')$ on the spin level which covers $b$.
Here it is important to assume that \(P_{\Spin}(M, g')\) corresponds to the same topological spin structure as~\(P_{\Spin}(M, g)\).
Let now \(S_g\) and \(S_{g'}\) be spinor bundles associated to \(P_{\Spin}(M, g)\) and \(P_{\Spin}(M, g')\) via the same representation \(\mu\colon \Spin\to SO(V)\).
Since $\tilde{b}$ is $\Spin(2)$-equivariant, it induces an isometry between the spinor bundles as Riemannian vector bundles, denoted by
\begin{equation}
	\beta\equiv \beta^g_{g'}\colon S_g\to S_{g'},
\end{equation}
which is to say, $ g_s(\cdot,\cdot)=g_s'\left(\beta(\cdot),\beta(\cdot)\right)$.
Moreover, note that for a vector $v\in TM$ and a spinor $\sigma\in S_g$,
\begin{equation}\label{eq:compatible isometries}
	\beta(\gamma(v)\sigma)=\gamma'(b(v))\beta(\sigma).
\end{equation}
where $\gamma'$ denotes the Clifford map with respect to the metric $g'$.

Next we consider the conformal behavior of the Dirac action.
First recall the following Proposition:
\begin{prop}[{see~\cite[Prop. 1.3.10]{ginoux2009dirac}}]%
\label{conformal transformation of dirac operator}
	Let $u\in C^\infty(M)$, then $g'=e^{2u}g$ is a Riemannian metric conformal to $g$.
	The spin Dirac operators \(\pd_g\) and \(\pd_{g'}\) satisfy
	\begin{equation}
		\pd_{g'}\beta(\sigma) =e^{-u}\beta\left(\pd_g \sigma+\frac{m-1}{2}\gamma(\grad(u))\sigma\right),
	\end{equation}
	where $m=\dim M$.
	Moreover, after a rescaling by $e^{-\frac{m-1}{2}u}$, the Dirac operator transforms homogeneously:
	\begin{equation}
		\pd_{g'}\beta(e^{-\frac{m-1}{2}u}\sigma)
		=e^{-\frac{m+1}{2}u}\beta(\pd_g\sigma).
	\end{equation}
\end{prop}
This implies that the dimension of the space of harmonic spinors is a conformal invariant.
Let us now specify to the surface case, $m=2$.
The map \(\beta\) allows us to transform the spinor according to the change of the metric.
However, by Proposition~\ref{conformal transformation of dirac operator} an additional rescaling is necessary to make the Dirac action invariant.
Indeed,
\begin{equation}\label{Dirac action-not conformal invariant}
	\begin{split}
		DA(\beta\sigma;e^{2u}g)
		&=\int_M e^{u}g_s(\sigma, \pd_g\sigma)\dv_g,
	\end{split}
\end{equation}
which is in general not equal to $DA(\sigma;g)$.
In contrast, if the rescaling by $e^{-\frac{m-1}{2}u}=e^{-\frac{1}{2}u}$ is taken into account the Dirac action is invariant:
\begin{equation}%
\label{eq:DiracActionRescaledConformalInvariance}
	DA(e^{-\frac{1}{2}u}\beta\sigma;e^{2u}g)
	=\int_M g_s'\left(e^{-\frac{1}{2}u}\beta\sigma, e^{-\frac{3}{2}u}\beta(\pd_g\sigma)\right) e^{2u}\dv_g
	=DA(\sigma;g).
\end{equation}
Hence, strictly speaking the Dirac action is not conformally invariant.
We will rather use the term “rescaled conformal invariance” for the invariance given in Equation~\eqref{eq:DiracActionRescaledConformalInvariance}.

\subsection{The energy-momentum tensor of the Dirac action}
We need to calculate the variation of the Dirac action with respect to the Riemannian metric.
To this end we need a parametrized version of the theory presented in the last subsection.
Let ${(g_t)}_t$ be a smooth family of Riemannian metrics parametrized by $t$ in a neighborhood of zero in $\R$ such that $g_0=g$.
The $t$-derivative at $t=0$ is a smooth symmetric $2$-form, say~$k\in \Gamma(\Sym(T^*M\otimes T^*M))$.
Conversely, any $k\in\Gamma(\Sym(T^*M\otimes T^*M))$ can be represented in this way, for instance, ${(g_t\coloneqq g+tk)}_t$ is such a family for $|t|$ small enough.

As before there is a unique family of self-adjoint endomorphisms \({(H_t)}_t\subset\End(TM)\) such that \(g_t(\cdot, \cdot) = g(H_t\cdot, \cdot)\).
We will denote by \(b_t\equiv b^g_{g_t}\) and \(\beta_t\equiv\beta^g_{g_t}\colon S_g\to S_{g_t}\).
Notice that
\begin{equation}\label{variation of frame}
	\dt{0} b_t= -\frac{1}{2} K\in \End(TM),
\end{equation}
where $K$ is the endomorphism associated to $k$ by metric dualization, which is also the $t$-derivative of $H_t$ at $t=0$.
Let $\{e_\alpha\}$ be a local oriented $g$-orthonormal frame, then $\{E_\al(t)=b_t(e_\al)\}$ is a local oriented $g_t$-orthonormal frame, and
\begin{equation}
	\dt{0} E_\al(t)=\dt{0} b_t(e_\al)=-\frac{1}{2} Ke_\al=-\frac{1}{2}K^\be_\al e_\be.
\end{equation}
We also need to consider the volume forms of different metrics: $\dv_{g_t}=\sqrt{\det H_t}\dv_g$.
The $t$-derivative at $t=0$ is
\begin{equation}
	\dt{0} \dv_{g_t} =\left(\dt{0}\sqrt{\det H_t}\right)\dv_g=\frac{1}{2}\tr_g(K) \dv_g.
\end{equation}

Transport the Dirac operator $\pd_{g_t}\colon \Gamma(S_{g_t})\to \Gamma(S_{g_t})$ via the isometry $\be_t$ to obtain a differential operator on~$\Gamma(S_g)$:
\begin{equation}
	\ov{\pd}_{g_t}\coloneqq \be_t^{-1}\circ \pd_{g_t}\circ \be_t\colon \Gamma(S_g)\to \Gamma(S_g).
\end{equation}
From~\cite{bourguignon1992spineurs} and~\cite{maier1997generic}, we know that
\begin{equation}
	\dt{0}\ov{\pd}_{g_t}=-\frac{1}{2} \gamma(e_\al)\na^s_{K(e_\al)}+\frac{1}{4} \gamma\left(\grad(\tr_g k)-{\diverg_g(k)}_\sharp\right),
\end{equation}
where the $g$-divergence $\diverg_g(k)$ is recalled in Section~\ref{sect:formal divergence}.

Now we are ready to derive the variation of the Dirac action for the metric $g$ while keeping the spinors essentially ``unchanged''.
The variation formula for metrics can be obtained in the following way:
\begin{equation}\label{Dirac action-energy-momentum}
	\begin{split}
		\dt{0} DA(\be_t\sigma;g_t)
		&=\dt{0} \int_M g_s(t)\left(\be_t\sigma,\pd_{g_t}\be_t\sigma\right)\dv_{g_t} \\
		&=\dt{0} \int_M g_s(\sigma,\be_t^{-1}\pd_{g_t}\be_t \sigma)\dv_{g_t} \\
		&=\dt{0} \int_M g_s(\sigma,\overline{\pd}_{g_t} \sigma) \dv_{g_t} \\
		&=-\frac{1}{2} \int_M \langle k, T(\sigma;g)\rangle\dv_g,
	\end{split}
\end{equation}
where $T(\sigma;g)=T_{\al\be}e^\al\otimes e^\be$ is the energy-momentum tensor, with coefficients
\begin{equation}
	T_{\al\be}
	=\frac{1}{2}\left\langle\sigma, \gamma(e_\al)\na^s_{e_\be}\sigma+\gamma(e_\be)\na^s_{e_\al}\sigma\right\rangle_{g_s}
		-\langle\sigma,\pd_g\sigma\rangle_{g_s} g_{\al\be}.
\end{equation}
By construction the energy-momentum tensor is symmetric, but not necessarily traceless; in fact,
\begin{equation}%
\label{Dirac action-trace of energy-momentum}
	\tr_g(T)=-\left\langle\sigma,\pd_g \sigma\right\rangle_{g_s}.
\end{equation}
This can be explained by the conservation law associated to the rescaled conformal invariance as follows.
\begin{equation}
	\begin{split}
		0 = \left.\frac{\dd}{\dd{t}}\right|_{t=0} DA(e^{-\frac12 tu}\beta \sigma, e^{2tu}g)
			&= \int_M 2\left<-\frac12 u \sigma, \pd\sigma\right> - \frac12\left<2u g, T\right> \dd{vol}_g \\
			&= -\int_M u\left(\left<\sigma, \pd\sigma\right> + \tr_g(T)\right) \dd{vol}_g
	\end{split}
\end{equation}
The conservation law to the rescaled conformal invariance prescribes the trace of the energy-momentum tensor.
As the Dirac action is in general not conformally invariant, the trace is non-zero.
Notice that on shell the energy-momentum tensor is traceless.

\subsection{Diffeomorphism invariance of the Dirac action}%
\label{sect:diffeomorphism invariance of Dirac action}
While diffeomorphism invariance of the Dirac action may seem obvious at first glance, we need a precise formulation to obtain the corresponding conservation law.

Let $f$ be a smooth diffeomorphisms of $M$.
Pull the metric $g$ back via $f$ to get a Riemannian metric $g_f$ on $TM$.
The differential $Tf$ of $f$ is an isometry of Riemannian vector bundles which covers the map $f$; that is, the following diagram commutes:
\begin{diag}
	\matrix[mat, column sep=small](m)
		{
		(TM, g_f) & (TM,g) \\
		M         & M      \\
		};
	\path[pf]
		(m-1-1) edge node{\(Tf\)} (m-1-2)
			edge (m-2-1)
		(m-2-1) edge node{\( f\)} (m-2-2)
		(m-1-2) edge (m-2-2)
	;
\end{diag}
This induces a map on the orthonormal frame bundles, which is also denoted by $T f$.
As $Tf$ is $\SO(2)$-equivariant, there exists a unique spin structure $P_{\Spin}(M,g_f)\to P_{\SO}(M,g_f)$ such that $Tf$ lifts to the corresponding principal $\Spin(2)$-bundles.
Indeed, as explained in~\cite[Theorem II.1.7]{lawson1989spin}, the spin structures on \(M\) are in one-to-one correspondence to elements of \(H^1(M, \mathbb{Z}_2)\) and \(P_{\Spin}(M, g_f)\) is given by the pullback of the cohomology class corresponding to \(P_{\Spin}(M, g)\).
Hence, in general, the topological spin structures corresponding to \(P_{\Spin}(M, g)\) and \(P_{\Spin}(M, g_f)\) do not need to coincide.
The situation is summarized the following commutative diagram in the left:
\begin{equation}%
\label{diag:PullbackSpinorBundle}
	\begin{tikzpicture}[commutative diagrams/.cd, every diagram]
		\matrix[mat,row sep=1.7em](m)
		{
			P_{\Spin}(M,g_f) & P_{\Spin}(M,g) \\
			P_{\SO}(M, g_f)  & P_{\SO}(M,g)  \\
			M                & M             \\
		};
		\path[pf]
		{
			(m-1-1) edge (m-2-1)
				edge node{\(\widetilde{Tf}\)} (m-1-2)
			(m-1-2) edge (m-2-2)
			(m-2-1) edge (m-3-1)
				edge node{\(Tf\)} (m-2-2)
			(m-2-2) edge (m-3-2)
			(m-3-1) edge node{\(f\)} (m-3-2)
		};
	\end{tikzpicture}
	\hspace{5em}
	\begin{tikzpicture}[commutative diagrams/.cd, every diagram]
		\matrix[mat,row sep=5em](m)
		{
			(S_{g_f},   g_s^f) & (S_g, g_s) \\
			M                  & M          \\
		};
		\path[pf]
		{
			(m-1-1) edge node{\(F\)} (m-1-2)
				edge (m-2-1)
			(m-2-1) edge node{\(f\)} (m-2-2)
			(m-1-2) edge (m-2-2)
		} ;
	\end{tikzpicture}
\end{equation}
An irreducible spin representation $\mu\colon\Spin(2)\to SO(V)$ will give rise to spinor bundles associated to the above principal $\Spin(2)$-bundles, and the isomorphism $\widetilde{Tf}$ induces an isometry $F$ of the corresponding spinor bundles, as shown in the commutative diagram above in the right.

In particular note that $F$ being an isometry means that for any $y\in M$ with $f(y)\equiv x\in M$, and for any $\sigma_1, \sigma_2\in {(S_g)}_x$,
\begin{equation}
	{g_s}_{|x}(\sigma_1,\sigma_2)={g_s^f}_{|y}\left(F_{|y}^{-1}\sigma_1,F_{|y}^{-1}\sigma_2\right).
\end{equation}
As a result, the Dirac operators $\pd_g$ on $S_g$ and $\pd_{g_f}$ on $S_{g_f}$ are the ``same'' up to the isometry $F$ in the sense that
\begin{equation}
	{(\pd_g)}_x=F_{|y} \circ {(\pd_{g_f})}_y \circ F^{-1}_{|y}.
\end{equation}
\begin{remark}
	One should note that this formula holds because $F$ is induced by an isometry $f\colon (M,g_f)\to (M,g)$.
	Then $Tf$ preserves the Levi-Civita connections on the tangent bundles, and hence $F$ preserves the spin connection.
	As a comparison, although the morphism~$\beta$ constructed in the previous subsection is also an isometry, it will not preserve the Dirac operator.
	Indeed, the map $b$ preserves the metrics but not the Lie brackets (this can be seen already in the case of a conformal perturbation of the Riemannian metric), hence $b$ does not preserve the Levi-Civita connection and consequently $\beta$ does not necessarily preserve the spin connections.
	This is the reason why a change of the Riemannian metric will give rise to a change of the Dirac operator, which we have used before.
\end{remark}

Now we explain the diffeomorphism invariance.
The claim is that
\begin{equation}\label{eq:diffeomorphism invariance of DA}
	DA(\sigma;g)=DA(F^{-1}\circ\sigma\circ f; g_f).
\end{equation}
Notice that $F^{-1}\circ \sigma \circ f$ is a section of the bundle $S_{g_f}\to (M, g_f)$, which is clear from the diagram~\eqref{diag:PullbackSpinorBundle}.
Then we have
\begin{equation}
	\begin{split}
		DA(F^{-1}\circ\sigma\circ f; g_f)
		&= \int_M {g_s^f}_{|y} \left(F^{-1}_{|y}{(\sigma\circ f)}_y, {(\pd_{g_f})}_y \left(F^{-1}_{|y}{(\sigma\circ f)}_y\right)\right) \dv_{g_f}(y) \\
		&= \int_M {g_s}_{|f(y)} \left(\sigma(f(y)), F_{|y}{(\pd_{g_f})}_y F^{-1}_{|y}(\sigma(f(y)))\right)\dv_{g_f}(y)\\
		&=\int_M {g_s}_{|x} \left(\sigma(x), {(\pd_g)}_x \sigma(x)\right) \dv_g(x) \\
		&= DA(\sigma;g).
	\end{split}
\end{equation}
Thus the claim~\eqref{eq:diffeomorphism invariance of DA} is confirmed.

To obtain the corresponding conservation law, we take a (local) one-parameter group of diffeomorphisms~$(f_t)$ of $M$ with $f_0=\id_M$ and
\begin{equation}
	\dt{0}f_t=X\in \Gamma(TM).
\end{equation}
For example, the flow generated by $X$ is such a family; the flow is global since $M$ is assumed to be compact.
Write $M_t=f_t^{-1}(M)=f_{-t}(M)$ and denote the pullback metrics on $M$ by $g_t\equiv g_{f_t}$.
The differential $Tf_t$ is again an isometry and hence can be viewed as a map of the principal $\SO(2)$-bundles.
Note that $M_t=M$ and hence $g$ is also a Riemannian metric on $TM_t$.
These two metrics can be related by an isometry $b_t\equiv b^g_{g_t}\colon (TM_t, g)\to (TM_t, g_t)$ as before and can be lift to the corresponding principal $\Spin(2)$-bundles.
In order to construct the lift, we notice that a diffeomorphism homotopic to the identity preserves the topological spin structure.
Therefore we have the following commutative diagram:
\begin{diag}
	\matrix[mat, column sep=small](m)
	{
		P_{\Spin}(M,g)    & P_{\Spin}(M_t,g_t) & P_{\Spin}(M,g) \\
		P_{\SO}(M,g)      & P_{\SO}(M_t, g_t)  & P_{\SO}(M,g)  \\
		M                 & M_t                & M             \\
		y                 & \id(y)=y           & x=f_t(y)      \\
	};
	\path[pf]
		(m-1-1) edge (m-2-1)
			edge node{\(\tilde{b}_t\)} (m-1-2)
		(m-1-2) edge (m-2-2)
			edge node{\(\widetilde{Tf}_t\)} (m-1-3)
		(m-1-3) edge (m-2-3)
		(m-2-1) edge (m-3-1)
			edge node{\(b_t\)} (m-2-2)
		(m-2-2) edge (m-3-2)
			edge node{\(Tf_t\)} (m-2-3)
		(m-2-3) edge (m-3-3)
		(m-3-1) edge node{\(\id\)} (m-3-2)
		(m-3-2) edge node{\(f_t\)} (m-3-3)
		(m-4-1) edge[commutative diagrams/mapsto] (m-4-2)
		(m-4-2) edge[commutative diagrams/mapsto] (m-4-3)
	;
\end{diag}
The bottom line exhibits the pointwise behavior of the maps on base manifolds.
Note that in this diagram all the horizontal maps are diffeomorphisms/isomorphisms.
The associated commutative diagram of spinor bundles is given by
\begin{diag}
	\matrix[mat, column sep=large](m)
	{
		(S_g, g_s) & (S_{g_t}, g_s(t)) & (S_g,g_s) \\
		M          & M_t               & M         \\
	};
	\path[pf]
		(m-1-1) edge node{\(\be_t\)} (m-1-2)
			edge (m-2-1)
		(m-2-1) edge node{\(\id\)}   (m-2-2)
		(m-1-2) edge node{\(F_t\)}   (m-1-3)
			edge (m-2-2)
		(m-2-2) edge node{\(f_t\)} (m-2-3)
		(m-1-3) edge (m-2-3)
	;
\end{diag}
The diffeomorphism invariance of the Dirac action says that for any $t$,
\begin{equation}
	DA(\sigma;g)=DA(F_t^{-1}\circ\sigma\circ f_t; g_t).
\end{equation}
Taking linearizations of both sides, we get
\begin{equation}
	\begin{split}
		0&=\dt{0}DA(F_t^{-1}\circ\sigma\circ f_t; g_{f_t}) \\
		&=\dt{0}\int_{M_t} {g_s(t)}_{|y} \left({F_t}_{|y}^{-1}{(\sigma\circ f_t)}_y, {(\pd_{g_t})}_y \left({F_t}_{|y}^{-1}{(\sigma\circ f_t)}_y\right)\right) \dv_{g_t}(y)\\
		&=\dt{0}\int_{M_t} {g_s(t)}_{|y} \left( \be_t\circ \be_t^{-1}\circ{F_t}_{|y}^{-1}{(\sigma\circ f_t)}_y, {(\pd_{g_t})}_y \left( \be_t\circ \be_t^{-1}\circ {F_t}_{|y}^{-1}{(\sigma\circ f_t)}_y\right)\right) \dv_{g_t}(y).
	\end{split}
\end{equation}
Write $\sigma_t\equiv \be_t^{-1}\circ {F_t}_{|y}^{-1}{(\sigma\circ f_t)}_y$.
Then using the spinor Lie-derivative as introduced in~\cite{bourguignon1992spineurs}, we have
\begin{equation}
	\dt{0}\sigma_t(y)=\dt{0}\be_t^{-1}\circ {F_t}_{|y}^{-1}{(\sigma\circ f_t)}_y=\Lie_X^S \sigma(y).
\end{equation}
Then, from the invariance~\eqref{eq:diffeomorphism invariance of DA} it follows that
\begin{equation}
	\begin{split}
		0&=\dt{0}\int_{M_t} {g_s(t)} \left(\be_t\sigma_t,\pd_{g_t}\be_t \sigma_t\right) \dv_{g_t} \\
		&=2\int_M \langle\Lie_X^S \sigma,\pd_g \sigma\rangle \dv_g
			-\frac{1}{2}\int_M \langle\Lie_X g, T\rangle\dv_g \\
		&=2\int_M \langle X, \diverg_\sigma(\pd_g\sigma)\rangle\dv_g
			+\int_M \langle X,\diverg_g(T)\rangle\dv_g.
	\end{split}
\end{equation}
Here the expression \(\diverg_\sigma\) is a formal analogue of the divergence for spinors, explained in Section~\ref{sect:formal divergence}.
As $X$ can be arbitrary, it follows that for any $\sigma$, the conservation law reads
\begin{equation}
	2\diverg_\sigma(\pd_g\sigma)+\diverg_g(T)=0
\end{equation}
In particular, along critical spinors which are the harmonic spinors, the energy-momentum tensor is divergence-free.
This confirms our expectation: the energy-momentum tensor is divergence-free on shell.

\subsection{Conclusions and remarks}
We summarize our discussion about the Dirac action in the following theorem.
\begin{thm}
	Let $S_g$ be a spinor bundle over $(M,g)$, and consider the following Dirac action defined on pure spinors $\sigma\in\Gamma(S_g)$ and Riemannian metrics $g$ by
	\begin{equation}
		DA(\sigma;g)=\int_M g_s(\sigma,\pd_g\sigma)\dv_g.
	\end{equation}
	\begin{enumerate}
		\item The total variation formula is
			\begin{equation}
				\delta DA=\int_M 2\langle \delta\sigma, EL(\sigma)\rangle - \frac12\langle\delta g, T(\sigma;g)\rangle\dv_g.
			\end{equation}
			where the Euler--Lagrange equation for the spinor is $EL(\sigma)=\pd_g\sigma=0$, and the energy-momentum tensor in a local oriented orthonormal frame $(e_\al)$ is
			\begin{equation}
				T(\sigma;g)=\left\{\frac{1}{2}\left\langle\sigma,\gamma(e_\al)\na^s_{e_\be}\sigma+\gamma(e_\be)\na^s_{e_\al}\sigma\right\rangle_{g_s}-\langle\sigma,\pd_g\sigma\rangle_{g_s} g_{\al\be}\right\}e^\al\otimes e^\be.
			\end{equation}
		\item This Dirac action is invariant under rescaled conformal transformations.
			That is, any $u\in C^\infty(M)$ induces a conformal metric $g'=e^{2u}g$ and there is a map $\beta\colon S_g\to S_{g'}$ such that $DA(e^{-\frac{1}{2}u}\beta\sigma;g')=DA(\sigma;g)$.
		\item The Dirac action is invariant under diffeomorphisms.
			That is, for any $f\in\Diff(M)$, there is an induced map $F\colon S_{g_f}\to S_g$  such that $DA(F^{-1}\circ \sigma\circ f;g_f)=DA(\sigma;g)$.
		\item The energy-momentum is a symmetric 2-tensor, with trace $\tr_g(T)=-\langle\sigma,\pd_g\sigma\rangle_{g_s}$.
			When~$\sigma$ is harmonic, $T$ is traceless and divergence-free, hence corresponds to a holomorphic quadratic differential on $M$.
	\end{enumerate}
\end{thm}
The methods we presented here do generalize to higher dimensions and all results except the rescaled conformal invariance do hold there.
It is also straightforward to adapt the theory presented here to models with mass-term or other terms which do not involve derivatives of spinor fields.

\section{The case of the nonlinear sigma model with gravitinos}%
\label{sect:pseudo supersymmetric action}
We now turn to the full nonlinear sigma model with gravitino as considered in~\cite{jost2016regularity}.
Let us first briefly recall its construction.
As before $\map\colon (M,g)\to(N,h)$ is a map from a Riemann surface to a Riemannian manifold, and a fixed spin structure $P_{\Spin}(M,g)\xrightarrow{\xi} P_{\SO}(M,g)$ has been chosen.
$(S_g, g_s)$ is a rank-four real a spinor bundle associated to this spin structure.

On the twisted vector bundle $S_g\otimes \map^*TN$ there is the induced metric $g_s\otimes\map^*h$ and induced metric connection, denoted by $\widetilde{\na}$.
A section \(\psi\) of this bundle, called a ``vector spinor'', will serve as a super partner of the map $\map$.
Note that a spin Dirac operator $\D_g$ can be defined in the same manner as above, i.e.,
\begin{equation}
	\D_g\psi=\gamma(e_\al)\widetilde{\na}_{e_\al}\psi.
\end{equation}
Let $\{y^i\}$ be a local coordinate of $N$, then $\psi$ can be locally written as $\psi=\psi^j\otimes\map^*\left(\frac{\p}{\p y^j}\right)$, where each $\psi^j$ is a (local) pure spinor.
The Dirac operator then acts as
\begin{equation}\label{eq:dfn:Dirac operator}
	\begin{split}
		\D_g\psi
		&=\gamma(e_\al)\left(\na^s_{e_\al}\psi^j\otimes \map^*\left(\frac{\p}{\p y^j}\right) +\psi^j\otimes\na^{\map^*TN}_{e_\al} \frac{\p}{\p y^j}\right) \\
		&=\pd_g\psi^j\otimes\map^*\left(\frac{\p}{\p y^j}\right)+\gamma(e_\al)\psi^j\otimes \map^*\left(\na^{TN}_{T\map(e_\al)} \frac{\p}{\p y^j}\right).
	\end{split}
\end{equation}

A super partner of the Riemannian metric $g$ is given by a gravitino $\chi$, which is a section of the bundle~$S_g\otimes TM$.
This bundle splits into two orthogonal summands, with the projections given by
\begin{align}
	P\chi=-\frac{1}{2}\gamma(e_\be)\gamma(e_\al)\chi^\al\otimes e_\be, & &
	Q\chi=-\frac{1}{2}\gamma(e_\al)\gamma(e_\be)\chi^\al\otimes e_\be.
\end{align}
For further geometric properties of gravitinos in super geometry we refer to~\cite{jost2014super}.

Then the action functional is given by
\begin{equation}
\label{eq:action functional}
	\begin{split}
		\A(\map, \psi;g, \chi)\coloneqq \int_M & |\dd \map|_{\gco\otimes \map^*h}^2
			+ \langle \psi, \D \psi \rangle_{g_s\otimes \map^*h}  \\
		&  -4\langle (\mathds{1}\otimes\map_*)(Q\chi), \psi \rangle_{g_s\otimes\map^*h}
			-|Q\chi|^2_{g_s\otimes g} |\psi|^2_{g_s\otimes \map^*h}
			-\frac{1}{6} \Rm^{N}(\psi) \dd vol_g,
	\end{split}
\end{equation}
where $\Rm^N$ stands for the pullback of the curvature of $N$ under $\map$, and the curvature term in the action is defined by
\begin{equation}
	-\frac{1}{6}\Rm^N(\psi)=-\frac{1}{6}\Rm^{N}_{ijkl}(\phi)\left\langle \psi^i,\psi^k\right\rangle_{g_s} \left\langle\psi^j,\psi^l\right\rangle_{g_s}.
\end{equation}

The Euler--Lagrange equation of \(\A\) are computed in~\cite{jost2016regularity}:
\begin{equation}%
\label{EL-eq}
	\begin{split}
		0 = EL(\map) &= \tau(\phi) - \frac{1}{2}\Rm^{\phi^*TN}(\psi, e_\al\cdot\psi)\phi_* e_\al + \frac{1}{12}S\na R(\psi)   \\
			&\quad + \langle \na^S_{e_\be}(e_\al \cdot e_\be \cdot \chi^\al), \psi \rangle_S
			+ \langle e_\al \cdot e_\be \cdot \chi^\al, \na^{S\otimes\phi^*TN}_{e_\be} \psi \rangle_S,  \\
		0 = EL(\psi) &= \D\psi - |Q\chi|^2\psi - \frac{1}{3}SR(\psi) - 2(\mathds{1}\otimes \phi_*)Q\chi
	\end{split}
\end{equation}
Here \(\tau(\phi)\) is the tension field of the map \(\phi\) and \(SR(\psi)\) is a term involving the curvature of the target and is of third order in \(\psi\).
The term \(S\nabla R(\psi)\) involves derivatives of \(\Rm^N\) and is of fourth order in \(\psi\).
For the precise form of the curvature terms we refer to~\cite{jost2016regularity}.

From the case of the Dirac action treated before and the analogous two-dimensional nonlinear super symmetric sigma model considered in super string theory (see for example~\cite{brink1976locally, deser1976complete}) that the model is invariant under (rescaled) conformal transformations, super Weyl transformations and diffeomorphisms.
We will see that this expectation is met.
The corresponding conservations laws will be formulated with the help of the energy-momentum tensor and the supercurrent which is given by the variation of the action with respect to the gravitino.
The rescaled conformal and super Weyl symmetry lead to algebraic properties of energy-momentum tensor and supercurrent respectively.
The diffeomorphism invariance leads to a coupled differential equation for the energy-momentum tensor and the supercurrent which holds on shell.
At first sight it might be surprising that the conservation law for the diffeomorphism invariance couples the energy-momentum tensor and the supercurrent.
Notice, however, that by “on shell” we only assume that the equations of motion for the map \(\map\) and the vector spinor \(\psi\) hold.

An important feature of the model considered in physics is super symmetry.
As the model introduced in~\cite{jost2016regularity} was built only with commuting variables, full supersymmetry cannot be expected.
However, we show that in special cases there is a remainder of super symmetry, which we call “degenerate supersymmetry”.
The conservation law associated to degenerate super symmetry leads to a differential equation for the supercurrent.
Together with the conservation law from diffeomorphism invariance we can conclude in this special case that energy-momentum tensor and supercurrent are holomorphic quantities.

One sees immediately that the Dirac-harmonic functional introduced in~\cite{chen2005regularity, chen2006dirac} and some of its variants~\cite{chen2008liouville,chen2008nonlinear} are special cases of this action functional.
Hence the computations presented here contain the several variants of the Dirac-harmonic action functional.
This leads to an interesting application of degenerate super symmetry to Dirac-harmonic maps, showing that two known simple solutions are actually related by degenerate super symmetry.

\subsection{Rescaled conformal symmetry and super Weyl symmetry}
As explained before the spin structure, spinor bundle and Dirac operator depend on the chosen metric \(g\).
If we now choose a metric \(g'=e^{2u}g\) that is conformal to \(g\) we also need the isometric bundle isomorphisms \(b\colon (TM, g)\to (TM, g')\) and \(\beta\colon (S_g, g_s)\to (S_{g'}, g_s')\) defined in Section~\ref{sect:ConformalInvarianceDiracAction}.
Furthermore, as in the case of the Dirac action an additional rescaling of the spinors is needed.
Hence we obtain the following rescaled conformal invariance:
\begin{equation}%
\label{eq:super action:rescaled conformal invariance}
	\A\left(\map,e^{-\frac12 u}(\beta\otimes \mathds{1})\psi;e^{2u}g,e^{-\frac12 u}(\beta\otimes b)\chi\right)=\A\left(\map,\psi;g,\chi\right).
\end{equation}
Indeed, under the conformal transformations, the volume form rescales as $\dv_{g'}=e^{2u}\dv_g$.
The harmonic term behaves as in the classical theory and the Dirac term is very similar to the case considered in Section~\ref{sect:ConformalInvarianceDiracAction}.
The three remaining terms can be checked easily using the fact that \(b\) and \(\beta\) are isometries.
For the third term it holds:
\begin{equation}
	\begin{split}
		-4\langle(\mathds{1}\otimes\map_*)&Q' e^{-\frac{1}{2}u}(\beta\otimes b)\chi, e^{-\frac{1}{2}u}(\beta\otimes\mathds{1})\psi\rangle_{g'_s\otimes\map^*h}\\
		&=2e^{-u}\left\langle\gamma'(b(e_\beta))\gamma'(b(e_\alpha))\beta(\chi^\beta)\otimes b(e_\alpha), \beta(\psi^j)\otimes \map^*\left(\frac{\p}{\p y^j}\right)\right\rangle_{g'_s\otimes \map^*h} \\
		&=2e^{-u}\left\langle\beta(\gamma(e_\be)\gamma(e_\al)\chi^\be)\otimes (e^{-u}e_\al), \beta(\psi^j)\otimes \map^*\left(\frac{\p}{\p y^j}\right)\right\rangle_{g'_s\otimes \map^*h} \\
		&=2e^{-2u}\left\langle \gamma(e_\be)\gamma(e_\al)\chi^\be\otimes e_\al, \psi\right\rangle_{g_s\otimes\map^*h} \\
		&=-4e^{-2u}\langle(\mathds{1}\otimes\map_*)Q\chi, \psi\rangle_{g_s\otimes\map^*h}.
	\end{split}
\end{equation}
For the fourth term notice
\begin{equation}
		|Q'(\beta\otimes b)e^{-\frac12 u}\chi |^2_{g_s'\otimes g'}
		= e^{- u} |(\beta\otimes b)Q\chi|^2_{g_s'\otimes g'}
		= e^{-u} |Q\chi|^2_{g_s\otimes g},
\end{equation}
and
\begin{equation}
	|(\beta\otimes\mathds{1})e^{-\frac12 u}\psi|^2_{g_s'\otimes\map^*h}
	= e^{-u} g_s'(\beta\psi^i,\beta\psi^j)h_{ij}(\map)
	= e^{-u} g_s(\psi^i,\psi^j)h_{ij}(\map)
	= e^{-u} |\psi|^2_{g_s\otimes\map^*h},
\end{equation}
which implies
\begin{equation}%
\label{eq:ConformalInvarianceIV}
	|Q'(\beta\otimes b)e^{-\frac12 u}\chi |^2_{g_s'\otimes g'}|(\beta\otimes\mathds{1})e^{-\frac12 u}\psi|^2_{g_s'\otimes\map^*h}
	= e^{-2u} |Q\chi|^2_{g_s\otimes g} |\psi|^2_{g_s\otimes\map^*h}.
\end{equation}
The curvature term again uses that \(\beta\) is an isometry:
\begin{equation}%
\label{eq:ConformalInvarianceV}
	\begin{split}
		\Rm((\beta\otimes\mathds{1})e^{-\frac12 u}\psi)
		&= e^{-2u} \Rm^{N}_{ijkl}g_s'\left(\beta\psi^i,\beta\psi^k\right)g_s'\left(\beta\psi^j,\beta\psi^l\right) \\
		&= e^{-2u} \Rm^{N}_{ijkl}g_s\left(\psi^i,\psi^k\right) g_s\left(\psi^j,\psi^l\right)
		= e^{-2u} \Rm(\psi).
	\end{split}
\end{equation}
This completes the verification of the rescaled conformal invariance.

In contrast to the rescaled conformal invariance, super Weyl invariance affects only the gravitino.
As only \(Q\chi\) enters the action functional, we have
\begin{equation}%
\label{eq:super action:SuperWeylInvariance}
	\A\left(\map, \psi; g, \chi + \zeta\right)=\A\left(\map,\psi;g,\chi\right).
\end{equation}
for \(Q\zeta=0\).
The property~\eqref{eq:super action:SuperWeylInvariance} is called super Weyl invariance.

We now take the opportunity to recall some properties of the bundle \(S_g\otimes TM\) and of the projections \(P\) and \(Q\), which have been presented in detail in~\cite{jost2016regularity}.
The Riemann surface \(M\) possesses an almost complex structure \(\ACM\) such that \(g(\ACM X, Y)=\dd{vol}_g(X, Y)\) for all vector fields \(X\) and \(Y\).
Left multiplication of spinors by the negative volume form \(-\omega=-e_1\cdot e_2\in \CliffordBundleM\) yields an almost complex structure on \(S_g\) that is compatible with \(\ACM\), that is \(\gamma(\ACM X) s = -\gamma(X) \omega s\) for all spinors \(s\).
The decomposition of \(S_g\) in even and odd part yields \(S_g=\SpinorBundleP\oplus\SpinorBundleP\), where both summands are isomorphic as associated vector bundles to \(\SpinPFB\).
The almost complex structure \(\omega\) restricts to \(\SpinorBundleP\) and the restriction will be denoted by \(\ACSBP\).
For the complex line bundle \(W=(\SpinorBundleP, \ACSBP)\) it holds that \(W\otimes_\C W= T^*M\).
With this preparation we can now identify the summands in \(S_g\otimes TM = \ker Q\oplus \Ima Q\) to be
\begin{align}
	\ker Q &= S_g = W\oplus W &
	\Ima Q &= S_g\otimes_\C TM = \left(W\otimes_\C W^\vee\otimes_\C W^\vee\right)\oplus\left(W\otimes_\C W^\vee\otimes_\C W^\vee\right)
\end{align}
Hence both \(\ker Q\) and \(\Ima Q\) are naturally holomorphic vector bundles.

\subsection{Supercurrent}
The supercurrent is the variation of the action with respect to the gravitino.
As the gravitino enters the action only algebraically, computation of the supercurrent is straightforward.
Fix $(\map,\psi)$ as well as the metric~$g$ and vary the gravitino via~$X(t)=X^\al(t)\otimes e_\al\in \Gamma(S_g\otimes TM)$ with $X(0)=\chi$ and
\begin{equation}
	\dt{0} X^\al(t)=\zeta^\al.
\end{equation}
Then
\begin{equation}
	\begin{split}
		\dt{0} \A(\map,\psi;& g,X(t))
		=\dt{0} \int_M -4\langle (\mathds{1}\otimes\map_*)(QX(t)), \psi \rangle_{g_s\otimes\map^*h}-|QX(t)|^2_{g_s\otimes g} |\psi|^2_{g_s\otimes \map^*h} \dv_g.
	\end{split}
\end{equation}
This can be computed as follows.
\begin{gather}
	\begin{split}
		\dt{0} \int_M &2\langle \gamma(e_\al) \gamma(e_\be) X^\al\otimes \map_* e_\be, \psi \rangle_{g_s\otimes\map^*h} \dv_g \\
		&= \int_M 2\left\langle \gamma(e_\al) \gamma(e_\be) \left(\frac{\dd}{\dd t} X^\al\right) \otimes \map_* e_\be, \psi \right\rangle_{g_s\otimes\map^*h} \dv_g \Big|_{t=0} \\
		&= \int_M 2\left\langle \left(\frac{\dd}{\dd t} X^\al\right), \langle \map_* e_\be, \gamma(e_\be) \gamma(e_\al) \psi \rangle_{\map^*h} \right\rangle_{g_s} \dv_g \Big|_{t=0} \\
		&= \int_M 2\left\langle \zeta^\al, \langle \map_* e_\be, \gamma(e_\be) \gamma(e_\al) \psi \rangle_{\map^*h} \right\rangle_{g_s}\dv_g,
	\end{split} \displaybreak[0]\\
	\begin{split}
		\dt{0} \int_M &\frac{1}{2}\langle X^\be,\gamma(e_\al) \gamma(e_\be) X^\al  \rangle_{g_s} |\psi|^2_{g_s\otimes\map^*h} \dv_g  \\
		&= \int_M \left\langle  \left(\frac{\dd}{\dd t} X^\al\right) ,\gamma(e_\al) \gamma(e_\be) X^\al \right\rangle_{g_s} |\psi|^2_{g_s\otimes\map^*h}\dv_g \Big|_{t=0}  \\
		&= \int_M \langle \zeta^\be,\gamma(e_\al) \gamma(e_\be) \chi^\al \rangle_{\gs} |\psi|^2_{g_s\otimes\map^*h} \dv_g.
	\end{split}
\end{gather}
Hence setting the \emph{supercurrent} to be $J=J^\al\otimes e_\al\in\Gamma(S_g\otimes TM)$ with
\begin{equation}\label{eq:supercurrent}
	J^\al = 2\langle \map_* e_\be, \gamma(e_\be) \gamma(e_\al) \psi \rangle_{\map^*h} +|\psi|^2 \gamma(e_\be) \gamma(e_\al) \chi^\be,
\end{equation}
we obtain
\begin{equation}
	\dt{0} \A(\map,\psi;g,X(t))=\int_M \langle\zeta, J\rangle_{g_s\otimes g} \dv_g.
\end{equation}

The conservation law associated to the super Weyl symmetry is obtained as follows.
For any $\zeta\in\Gamma(S_g\otimes TM)$, we have \(QP\zeta=0\) and hence
\begin{equation}
	0 = \dt{0} \A(\map,\psi;g,\chi+tP\zeta) = \int_M \langle P\zeta, J\rangle =\int_M \langle \zeta, PJ\rangle_{g_s\otimes g} \dv_g
\end{equation}
Since $\zeta$ can be arbitrary, we conclude that $J$ satisfies
\begin{equation}
	PJ=0,
\end{equation}
and hence \(J\in\Gamma(\Ima Q)\).

\begin{remark}
	Note that for a section $\zeta=\zeta^\al\otimes e_\al\in\Gamma(S_g\otimes TM)$, the followings are equivalent:
	\begin{align}
		P\zeta=0& \Leftrightarrow &
		\gamma(e_\al)\zeta^\al=0 &\Leftrightarrow&
		\zeta^1=\gamma(e_1)\gamma(e_2)\zeta^2 &\Leftrightarrow&
		\zeta^2=-\gamma(e_1)\gamma(e_2)\zeta^1,
	\end{align}
	which are all equivalent to say that $\zeta=\zeta^\al\otimes e_\al$ lies in $\Ima Q=(\SpinorBundleM_g,\ACSBP\oplus\ACSBP) \otimes_{\C} TM$, where $\ACSBP$ is the complex structure on $\SpinorBundleP$.
	Note that the bundle \(\Ima Q\) is a  holomorphic vector bundle over the Riemann surface $M$.
	Later we will show that in some particular cases $J$ is actually an holomorphic section.
\end{remark}

\subsection{Energy-momentum tensor}
Let ${(g_t)}_t$ be a family of Riemannian metrics on $M$ with
\begin{equation}
	\dt{0}g_t=k\in \Gamma(\Sym(T^*M\otimes T^*M)),
\end{equation}
and let $K\in \End(TM)$ be the associated endomorphism.
As in the discussions of Dirac actions, the spinor bundles change theoretically with the metrics, so the vector spinors and gravitinos need to be carried using isometries along with the variation of metrics.
Thus, one needs to calculate the linearization
\begin{equation}
	\dt{0} \A\left(\map,(\be_t\otimes \mathds{1})\psi; g_t,(\be_t\otimes b_t)\chi\right)=\dt{0}\int_M\mathrm{I+II+III+IV+V}\dd{vol}_{g_t},
\end{equation}
where the roman numerals $\mathrm{I},\dots,\mathrm{V}$ denote the summands under integral of the action functional.
We calculate them as follows.
\begin{enumerate}[label=(\roman*)]
	\item\label{item:EMTHarmonic}
		The energy of the map can be analyzed as usual:
		\begin{equation}
			\dt{0}\mathrm{I}
			=\dt{0} E_{\al}(\map^i) E_\al(\map^j) h_{ij}(\map) \\
			=-\frac{1}{2} K^\be_\al \cdot\left(2\langle \map_*e_\al,\map_* e_\be\rangle_{\map^*h}\right).
		\end{equation}
	\item
		Locally write the vector spinor as $\psi=\psi^j\otimes \map^*\left(\frac{\p}{\p y^j}\right)$, then
		\begin{equation}
			\left(\be_t\otimes \mathds{1}\right)\psi=\be_t(\psi^j)\otimes \map^*\left(\frac{\p}{\p y^j}\right)
		\end{equation}
		is a section of $S_{g_t}\otimes \map^*TN$.
		Recall the definition of the twisted Dirac operator~\eqref{eq:dfn:Dirac operator},
		\begin{equation}
			\begin{split}
				\D_{g_t}(\be_t\otimes\mathds{1})\psi
				=\pd_{g_t}(\be_t\psi^j)\otimes\map^*\left(\frac{\p}{\p y^j}\right)
				+\gamma_t\left(E_\al(t)\right)\be_t\psi^j\otimes \map^*\left(\na^{TN}_{T\map(E_\al)}\frac{\p}{\p y^j}\right)
			\end{split}
		\end{equation}
		and consequently
		\begin{equation}
			\begin{split}
				\big\langle(\be_t\otimes \mathds{1})&\psi,
				\D_{g_t}(\be_t\otimes\mathds{1})\psi\big\rangle_{g_s(t)\otimes\map^*h}  \\
				=&g_s(t)\left(\be_t\psi^i,\pd_{g_t}(\be_t\psi^j)\right)h_{ij}(\map) \\
				&\qquad	+ g_s(t)\left(\be_t\psi^i,\gamma_t(E_\al)\be_t\psi^j\right)h\left(\frac{\p}{\p y^i}, \na^{TN}_{T\map(E_\al)}\frac{\p}{\p y^j}\right)\\
				=&g_s(\psi^i,\be_t^{-1}\pd_{g_t}(\be_t\psi^j)) h_{ij}(\map)
					+g_s(\psi^i,\gamma(e_\al)\psi^j)h\left(\frac{\p}{\p y^i}, \na^{TN}_{T\map(E_\al)}\frac{\p}{\p y^j}\right) \\
				=&g_s(\psi^i,\ov{\pd}_{g_t}\psi^j) h_{ij}(\map)
					+g_s(\psi^i,\gamma(e_\al)\psi^j)h\left(\frac{\p}{\p y^i}, \na^{TN}_{T\map(E_\al)}\frac{\p}{\p y^j}\right).
			\end{split}
		\end{equation}
		Taking derivative with respect to $t$  gives
		\begin{equation}
			\begin{split}
				\dt{0} \mathrm{II}
				&= \dt{0}\big\langle(\be_t\otimes\mathds{1})\psi, \D_{g_t}(\be_t\otimes\mathds{1})\psi\big\rangle_{g_s(t)\otimes\map^*h}\\
				&= g_s\left(\psi^i,\dt{0}\ov{\pd}_{g_t}\psi^j\right) h_{ij}(\map)\\
				&\qquad\qquad+g_s(\psi^i,\gamma(e_\al)\psi^j)h\left(\frac{\p}{\p y^i}, \dt{0}\na^{TN}_{T\map(E_\al)}\frac{\p}{\p y^j}\right) \\
				&=g_s\left(\psi^i,-\frac{1}{2}\gamma(e_\al)\na^s_{K(e_\al)}\psi^j\right) h_{ij}(\map) \\
					&\qquad\qquad  +g_s(\psi^i,\gamma(e_\al)\psi^j)h\left(\frac{\p}{\p y^i}, \na^{TN}_{T\map(-\frac{1}{2}K (e_\al))}\frac{\p}{\p y^j}\right)  \\
				&=-\frac{1}{2}K^\be_\al\left(g_s(\psi^i,\na^s_{e_\be}\psi^j)h_{ij}(\map)
					+g_s(\psi^i,\gamma(e_\al)\psi^j)h\left(\frac{\p}{\p y^i},\na^{TN}_{T\map(e_\be)}\frac{\p}{\p y^j}\right)\right)\\
				&=-\frac{1}{2}K^\be_\al \left\langle\psi,\gamma(e_\al)\widetilde{\na}_{e_\be}\psi\right\rangle_{g_s\otimes\map^*h} \\
				&=-\frac{1}{2}K^\be_\al \left\langle \psi, \frac{1}{2}\big(\gamma(e_\al)\widetilde{\na}_{e_\be}\psi+\gamma(e_\be)\widetilde{\na}_{e_\al}\psi\big)\right\rangle_{g_s\otimes\map^*h}.
			\end{split}
		\end{equation}
		where the second equality follows from the fact that for any vector field $X\in \Gamma(TM)$,
		\begin{equation}
			g_s(\psi^i,\gamma(X)\psi^j)h_{ij}(\map)=0
		\end{equation}
		due to the skew-adjointness of the Clifford multiplications, and the last equality holds since $K$ is symmetric.
	\item
		Next we consider the third summand of the action functional, which mixes all the fields together.
		Actually, since
		\begin{equation}
			\begin{split}
				-4\langle(\mathds{1}\otimes\map_*)&Q_t(\be_t\otimes b_t)\chi,
				(\be_t\otimes\mathds{1})\psi\rangle_{g_s(t)\otimes\map^*h} \\
				&=2\left\langle\gamma_t(E_\al)\gamma_t(E_\be)\be_t\chi^\al\otimes\map_*(E_\be), \be_t\psi^j\otimes\map^*\left(\frac{\p}{\p y^j}\right)\right\rangle_{g_s(t)\otimes\map^*h}\\
				&=2g_s(t)\left(\be_t(\gamma(e_\al)\gamma(e_\be)\chi^\al), \be_t\psi^j\right) h(\map)\left(T\map(E_\be),\frac{\p}{\p y^j}\right) \\
				&=2g_s\big(\gamma(e_\al)\gamma(e_\be)\chi^\al, \psi^j\big) h(\map)\left(T\map(b_t e_\be), \frac{\p}{\p y^j}\right),
			\end{split}
		\end{equation}
		we have
		\begin{equation}
			\begin{split}
				\dt{0}\mathrm{III}
				&= -4\dt{0} \langle(\mathds{1}\otimes\map_*)Q_t(\be_t\otimes b_t)\chi, (\be_t\otimes\mathds{1})\psi\rangle_{g_s(t)\otimes\map^*h}  \\
				&= 2g_s\big(\gamma(e_\al)\gamma(e_\be)\chi^\al,\psi^j\big) h(\map)\left(\dt{0}T\map(b_t e_\be), \frac{\p}{\p y^j}\right) \\
				&=-\frac{1}{2} K^\be_\al\cdot \langle\gamma(e_\eta)\gamma(e_\al)\chi^\eta\otimes \map_*e_\be
					+\gamma(e_\eta)\gamma(e_\be)\chi^\eta\otimes\map_* e_\al, \psi\rangle_{g_s\otimes\map^*h} \\
			\end{split}
		\end{equation}
	\item
		For the fourth term we conclude as in~\eqref{eq:ConformalInvarianceIV} that
		\begin{equation}
			\begin{split}
				\dt{0}\mathrm{IV}
				&= -\dt{0} |Q_t(\be_t\otimes b_t)\chi|^2_{g_s(t)\otimes g_t} |(\be_t\otimes \mathds{1})\psi|^2_{g_s(t)\otimes\map^*h} \\
				&= \dt{0} |Q\chi|^2_{g_s\otimes g} |\psi|^2_{g_s\otimes\map^*h}
				=0
			\end{split}
		\end{equation}
	\item
		The curvature term is analogous to~\eqref{eq:ConformalInvarianceV}:
			\begin{equation}
				\dt{0}\mathrm{V}
				= -\frac16 \dt{0} \Rm((\be_t\otimes\mathds{1})\psi)
				= -\frac16 \dt{0} \Rm(\psi)
				= 0
			\end{equation}
	\item\label{item:EMTVolumeForm}
		We still need to consider the change in the volume form.
		As in Section~\ref{sect:ConformalInvarianceDiracAction}, we have that
		\begin{equation}
			\dt{0} \dd{vol}_{g_t} = \frac12 \tr_g(K)\dd{vol}_g =\frac12 K_\alpha^\beta\delta_\beta^\alpha \dd{vol}_g
		\end{equation}
		Consequently we have
		\begin{equation}
			\int_M \left(\mathrm{I+II+III+IV+V}\right) \frac{\dd}{\dd t}\dd{vol}_{g_t}\Big|_{t=0} = \frac12\int_M \left(\mathrm{I+II+III+IV+V}\right)\Big|_{t=0} K_\alpha^\beta \delta_\beta^\alpha \dd{vol}_g.
		\end{equation}
\end{enumerate}
Summing the contributions from~\ref{item:EMTHarmonic} to~\ref{item:EMTVolumeForm} up we obtain
\begin{equation}
	\begin{split}
		\dt{0} \A\left(\map,(\be_t\otimes \mathds{1})\psi; g_t,(\be_t\otimes b_t)\chi\right)
		=-\frac{1}{2}\int_M K^\be_\al T^\al_\be \dv_g
		=-\frac{1}{2}\int_M \left\langle \frac{\p g_t}{\p t}\Big|_{t=0}, T\right\rangle \dv_g.
	\end{split}
\end{equation}
Here the inner product under integral is the induced one on symmetric covariant 2-tensors, $T^\al_\be= g^{\al\eta}T_{\eta\be}$ and \(T=T_{\alpha\beta}e^\alpha\otimes e^\beta\), where
\begin{equation}%
\label{eq:super action:energy-momentum}
	\begin{split}
		T_{\al\be}
		&=2\langle \map_*e_\al,\map_* e_\be\rangle_{\map^*h} + \frac{1}{2}\left\langle\psi,\gamma(e_\al)\widetilde{\na}_{e_\be}\psi+\gamma(e_\be)\widetilde{\na}_{e_\al}\psi\right\rangle_{g_s\otimes\map^*h} \\
		&\quad +\langle\gamma(e_\eta)\gamma(e_\al)\chi^\eta\otimes \map_*e_\be+\gamma(e_\eta)\gamma(e_\be)\chi^\eta\otimes\map_*e_\al, \psi\rangle_{g_s\otimes\map^*h}\\
		&\quad - \left(|\dd\map|^2_{\gco\otimes \map^*h} + \langle\psi,\D_g\psi\rangle - 4\langle(\mathds{1}\otimes\map_*)Q\chi, \psi\rangle - |Q\chi|^2 |\psi|^2 -\frac{1}{6}\Rm(\psi)\right)g_{\al\be}.
	\end{split}
\end{equation}
Being a quantity rising from the variation of a symmetric 2-tensor, the energy-momentum tensor is naturally symmetric, as the expression clearly shows.

One can verify that the energy-momentum tensor is in general not traceless.
This is due to the fact that the action functional is not invariant under conformal transformations on $g$, that is, for $g'=e^{2u}g$ in general
\begin{equation}
	\A\left(\map,(\beta\otimes \mathds{1})\psi; g',(\beta\otimes b)\chi\right)\neq \A\left(\map,\psi;g,\chi\right).
\end{equation}
Instead, the action functional is invariant under the rescaled conformal invariance, see~\eqref{eq:super action:rescaled conformal invariance}.
As in the case of the Dirac action the conservation law corresponding to the rescaled conformal invariance prescribes the trace of the energy-momentum tensor:
\begin{equation}
	\begin{split}
		0 &= \dt{0} \A\left(\map, e^{-\frac12 tu}(\beta_t\otimes \mathds{1})\psi; e^{2tu}g, e^{-\frac12 tu}(\beta_t\otimes b_t)\chi\right) \\
		&= \int_M 2\left< -\frac12 u \psi, EL(\psi)\right> - \frac12\left<2u g, T\right> + \left<-\frac12 u \chi, J\right> \dd{vol}_g \\
		&= -\int_M u\left(\tr_g(T) + \left<\psi, EL(\psi)\right> + \frac12\left<\chi, J\right>\right) \dd{vol}_g
	\end{split}
\end{equation}
As the integral has to vanish for all functions \(u\), we conclude
\begin{equation}
	\tr_g(T) = - \left<\psi, EL(\psi)\right> - \frac12\left<\chi, J\right>,
\end{equation}
where \(EL(\psi)\) denotes the Euler--Lagrange equation for \(\psi\) computed in~\cite{jost2016regularity}.
Notice, if the Euler--Lagrange equation for \(\psi\) is satisfied and either \(\chi\) or \(J\) vanish, then \(T\) is actually traceless.
In that case \(T\) can be identified with a smooth section of \(T^*M\otimes_\C T^*M\), that is a quadratic differential.
We will later show that under certain conditions \(T\) is actually a holomorphic quadratic differential.

\subsection{Diffeomorphism invariance}
It will be shown that the super action functional is invariant under diffeomorphisms if the fields are “pulled back” in an appropriate way.
Let $f\in \Diff(M)$ and consider the following \emph{diffeomorphism transformations}:
\begin{equation}\label{eq:super action:diffeomorphism transformation}
	\begin{split}
		\map &\mapsto \map'\coloneqq\map\circ f, \\
		\psi=\psi^j\otimes\map^*\left(\frac{\p}{\p y^j}\right) &\mapsto \psi'\coloneqq F^{-1}\circ \psi^j\circ f\otimes {(\map\circ f)}^*\left(\frac{\p}{\p y^j}\right), \\
		g& \mapsto g'\coloneqq g_f, \\
		\chi=\chi^\al\otimes e_\al &\mapsto \chi'\coloneqq F^{-1}\circ \chi^\al\circ f\otimes {(Tf)}^{-1}e_\al,
	\end{split}
\end{equation}
where $F\colon (S_{g_f}, \gs^f)\to (S_g,\gs)$ is the isomorphism introduced in Section~\ref{sect:diffeomorphism invariance of Dirac action}.
Then we claim that
\begin{equation}
	\A(\map',\psi';g',\chi')=\A(\map,\psi;g,\chi).
\end{equation}
To see this, suppose that under the diffeomorphism $f$, $y\mapsto x=f(y)$. Then as in the harmonic map case
\begin{equation}
	|\dd\map'|^2_{g'^\vee\otimes \map'^*h}(y)=|\dd\map|^2_{\gco\otimes\map^*h}(x).
\end{equation}
For those terms involving spinors, we note that for any spinor $\sigma\in\Gamma(S)$,
\begin{equation}\label{eq:equivariance of F_t}
	F^{-1}{\left(\gamma(e_\al)\sigma\right)}_{f(y)}=\gamma'\big({(Tf)}_y^{-1}e_\al\big)F^{-1}{(\sigma)}_{f(y)}
\end{equation}
where $\gamma'$ denotes the Clifford multiplications with respect to the metric $g'=g_f$.
From this we will see that the other terms are also invariant:
\begin{enumerate}[label=(\roman*)]
	\item
		First consider the Dirac term
		\begin{equation}
			\begin{split}
				\D_{g'}\psi'(y)
				&= {(\pd_{g'})}_y {F}_{|y}^{-1}{(\psi^k\circ f)}_y\otimes\map'^*\left(\frac{\p}{\p y^k}\right) \\
				&\qquad+\gamma'(Tf^{-1}(e_\al)){F}_{|y}^{-1}{(\psi^k\circ f)}_y\otimes \na^{TN}_{(T\map'){(Tf)}^{-1} e_\al}\left(\frac{\p}{\p y^k}\right)\\
				&={F}_{|y}^{-1}{(\pd_g)}_x \psi^k(x)\otimes \map'^*\left(\frac{\p}{\p y^k}\right) + {F}_{|y}^{-1}\left(\gamma(e_\al)\psi^k(x)\right)\otimes \na^{TN}_{T\map(e_\al)}\left(\frac{\p}{\p y^k}\right).
			\end{split}
		\end{equation}
		Thus
		\begin{equation}
		\begin{split}
			\langle \psi'&,\D_{g'}\psi'\rangle(y) \\
			&={g^f_s}_{|y}\left({F}_{|y}^{-1}{(\psi^j\circ f)}_y,{F}_{|y}^{-1}{(\pd_g)}_x \psi^k(x)\right) {h(\map'(y))}_{jk} \\
			&\quad+{g^f_s}_{|y}\left({F}_{|y}^{-1}{(\psi^j\circ f)}_y, {F}_{|y}^{-1}\left(\gamma(e_\al)\psi^k(x)\right)\right) h_{|\map'(y)}\left(\frac{\p}{\p y^j}, \na^{TN}_{T\map(e_\al)}\frac{\p}{\p y^k}\right) \\
			&={g_s}_{|x}\left(\psi^j(x),{(\pd_g)}_x\psi^k(x)\right) {h(\map(x))}_{jk} \\
			&\quad +{g_s}_{|x}\left(\psi^j(x),\gamma(e_\al)\psi^k(x)\right)h_{|\map(x)}\left(\frac{\p}{\p y^j}, \na^{TN}_{T\map(e_\al)}\left(\frac{\p}{\p y^k}\right)\right) \\
			&=\langle\psi,\D_g\psi\rangle(x).
		\end{split}
		\end{equation}
		Moreover,
		\begin{equation}
			\begin{split}
				|\psi'(y)|^2_{\gs^f\otimes\map'^*h}
				=&{\gs^f}_{|y}\left({F}_{|y}^{-1}{(\psi^j\circ f)}_y,F_{|y}^{-1}{(\psi^k\circ f)}_y\right) {h(\map'(y))}_{jk} \\
				=&{g_s}_{|x}\left(\psi^j(x),\psi^k(x)\right) {h(\map(x))}_{jk} \\
				=&|\psi(x)|^2_{g_s\otimes\map^*h}
			\end{split}
		\end{equation}
		and
		\begin{equation}
			\begin{split}
				{\Rm(\psi')}_{|y}
				&=R^N_{ijkl}(\map'(y)) {\gs^f}_{|y}\left({F}_{|y}^{-1}{(\psi^i\circ f)}_y,{F}_{|y}^{-1}{(\psi^k\circ f)}_y\right) \\
				&\qquad\times{\gs^f}_{|y}\left({F}_{|y}^{-1}{(\psi^j\circ f)}_y,{F}_{|y}^{-1}{(\psi^l\circ f)}_y\right) \\
				&={R^N(\map(x))}_{ijkl}{g_s}_{|x}\left(\psi^i(x),{(\pd_g)}_x\psi^k(x)\right)
				{g_s}_{|x}\left(\psi^j(x), {(\pd_g)}_x\psi^l(x)\right) \\
				&={\Rm(\psi)}_{|x}.
			\end{split}
		\end{equation}
	\item For the gravitino, from~\eqref{eq:equivariance of F_t} it follows that
		\begin{equation}
			Q'\chi'(y)={F}_{|y}^{-1}{\left(\gamma(e_\al)\gamma(e_\be)\chi^\al\right)}_{|f(y)}
			\otimes {(Tf)}^{-1} e_\al.
		\end{equation}
		Hence we have
		\begin{equation}
			|Q' \chi'(y)|^2_{\gs^f\otimes g_t}=|Q\chi(x)|^2_{g_s\otimes g}.
		\end{equation}
		For the mixed term, note that
		\begin{equation}
			\begin{split}
				(\mathds{1}\otimes\map'_{*})Q'\chi'(y)
				&=\gamma'\big({(Tf)}_y^{-1}e_\al\big)\gamma'\big({(Tf)}_y^{-1}e_\be\big){F}_{|y}^{-1}\chi^\al(f(y))
					\otimes\map'_{*} {(Tf)}^{-1} e_\al \\
				&={F}_{|f(y)}^{-1} {\left(\gamma(e_\al)\gamma(e_\be)\chi^\al\right)}_{f(y)}\otimes \map_{*} e_\al. \\
			\end{split}
		\end{equation}
		Then it is immediate that
		\begin{equation}
			\langle(\mathds{1}\otimes\map'_{*})Q'\chi'(y),\psi'(y)\rangle_{\gs^f\otimes\map'^*h}
			=\langle(\mathds{1}\otimes\map_*)Q\chi(x),\psi(x)\rangle_{g_s\otimes\map^*h}.
		\end{equation}
\end{enumerate}
Therefore, by the change of variable formula, the diffeomorphism invariance of \(\A\) is confirmed.

The symmetry of diffeomorphism invariance will give another conservation law\index{conservation law}.
Actually, let $X\in\Gamma(TM)$ generate a global flow $(f_t)$ as in Section~\ref{sect:diffeomorphism invariance of Dirac action}.
Then we have
\begin{equation}%
\label{eq:variation:diffeomorphism transformations}
	\begin{split}
		0&= \dt{0} \A(\map_t,\psi_t;g_t,\chi_t)\\
		&=\int_M -2\left\langle \dt{0}\map_t, EL(\map)\right\rangle + 2\left\langle \nabla^{S_g\otimes\phi_t^*TN}_{\p_t}\left((\beta_t^{-1}\otimes\mathds{1})\psi_t\right)\Big|_{t=0}, EL(\psi)\right\rangle\dv_g \\
			&\qquad+\int_M -\frac{1}{2}\left\langle\dt{0} g_t,T\right\rangle+\left\langle\dt{0}\chi_t, J\right\rangle \dv_g \\
		&=\int_M -2\left\langle \phi_*(X),EL(\map)\right\rangle+2\left\langle \delta\psi(X), EL(\psi)\right\rangle\dv_g  \\
		&\qquad\qquad	+\int_M -\frac{1}{2}\left\langle\Lie_X g, T\right\rangle+\left\langle\Lie^{S_g\otimes TM}_X \chi, J\right\rangle \dv_g.\\
	\end{split}
\end{equation}
Note that for the vector spinor we have to take the covariant derivative to obtain the variation field which is abbreviated as $\delta\psi(X)$ in the last formula, which is the same approach we took in~\cite{jost2016regularity}.
Here the Lie derivative on $S_g\otimes TM$ is the one defined in~\cite{bourguignon1992spineurs}; that is, under a local orthonormal frame $(e_\al)$, for $\chi=\chi^\al\otimes e_\al$,
\begin{equation}
	\Lie^{S_g\otimes TM}_X\chi\coloneqq
	\dt{0} \be_t^{-1}F_t^{-1}(\chi^\al\circ f_t)\otimes b_t^{-1} {(Tf_t)}^{-1}(e_\al\circ f_t).
\end{equation}
Notice, that on the vector part of $\chi$ the above $t$-derivative will not result in the ordinary Lie derivative on tangent vectors, see also~\cite{bourguignon1992spineurs}.
From Section~\ref{sect:formal divergence} one knows that
\begin{equation}
	\int_M \langle \Lie_X g,T\rangle\dv_g=-2\int_M \langle X,\diverg_g(T)\rangle\dv_g
\end{equation}
and
\begin{equation}
	\int_M \langle\Lie_X^{S_g\otimes TM} \chi, J\rangle\dv_g=\int_M \langle X, \diverg_\chi(J)\rangle\dv_g.
\end{equation}
Therefore, along solutions of the Euler--Lagrange equations, the following identity holds:
\begin{equation}
	\diverg_g(T)+\diverg_\chi(J)=0,
\end{equation}
where the formal divergence operator $\diverg_\chi$ is defined in Section~\ref{sect:formal divergence}.

\begin{remark}
	If the gravitino vanishes, then from~\eqref{eq:variation:diffeomorphism transformations} we know that, along solutions of the Euler--Lagrange equations,
	\begin{equation}
		0=\int_M -\frac{1}{2}\langle \Lie_X g,T\rangle\dv_g=\int_M  \langle X, \diverg_g T\rangle\dv_g.
	\end{equation}
	This tells us that $T$ is divergence-free, and hence the energy-momentum tensor corresponds to a holomorphic quadratic differential.
\end{remark}

Before going to the discussion on super symmetry, we summarize the results obtained up to now in the following theorem.
\begin{thm}%
\label{thm:SymmetryaboutActionwithGravitino}
	Consider the super action functional defined by~\eqref{eq:action functional}.
	\begin{enumerate}
		\item The total variation formula is
			\begin{equation}
				\delta\A
				=\int_M -2\langle\delta\map, EL(\map)\rangle+2\langle\delta\psi,EL(\psi)\rangle
					-\frac{1}{2}\langle\delta g,T\rangle+\langle\delta\chi, J\rangle \dv_g,
			\end{equation}
			where the Euler--Lagrange equations are given in~\eqref{EL-eq} and the energy-momentum tensor $T$ is given by~\eqref{eq:super action:energy-momentum} and the supercurrent \(J\) is given by~\eqref{eq:supercurrent}.
		\item This action functional is invariant under rescaled conformal transformations:
			\begin{equation}%
				\A\left(\map,e^{-\frac12 u}(\beta\otimes \mathds{1})\psi;e^{2u}g,e^{-\frac12 u}(\beta\otimes b)\chi\right)=\A\left(\map,\psi;g,\chi\right).
			\end{equation}
			Consequently, $\tr_g(T)=- \left<\psi, EL(\psi)\right> - \frac12\left<\chi, J\right>$, even off shell.
		\item This action is invariant under super Weyl transformations.
			That is, for any $\zeta\in \Gamma(S_g\otimes TM)$, it holds that $\A(\map,\psi;g,\chi+P\zeta)=\A(\map,\psi;g,\chi)$.
			Consequently, $PJ=0$, or equivalently, \(J\) is a smooth section of \(S_g\otimes_\C TM\).
		\item This action functional is invariant under diffeomorphisms.
			That is, it is invariant under the transformation~\eqref{eq:super action:diffeomorphism transformation} for $f\in \Diff(M)$.
			Consequently, along solutions of the Euler--Lagrange Equations~\eqref{EL-eq}, the coupled conservation law holds
			\begin{equation}\label{eq:coupled conservation law}
				\diverg_g(T)+\diverg_\chi(J)=0.
			\end{equation}
		\item If either $\chi= 0$ or $J=0$, then along solutions of the Euler--Lagrange equations, the energy-momentum tensor $T$ is symmetric, traceless and divergence-free.
			Hence it corresponds to a holomorphic quadratic differential on $M$.
	\end{enumerate}
\end{thm}

\subsection{Degenerate super symmetry}\label{sect:degenerate super symmetry}
The action functional~\eqref{eq:action functional} is motivated from the action functional of two-dimensional super symmetric sigma models~\cite{brink1976locally} also called the super conformal action functional in~\cite{jost2014super}.
The major reason for the introduction of those models was super symmetry.
As was argued in~\cite{kessler2016functional} super symmetry requires anti-commutative variables.
Hence a full super symmetry cannot be expected for the action functional~\eqref{eq:action functional}.
Surprisingly, the following special case of super symmetry, which we will call degenerate super symmetry, persists:
\begin{prop}%
\label{prop:DegenerateSUSY}
	Let \(q\) be a section of \(S_g\).
	The action functional~\eqref{eq:action functional} is invariant under the following infinitesimal transformations
	\begin{align}
		\delta\map &= \left<q, \psi\right>_{\gs} &
		\delta\psi&= -\gamma(\grad\map)q \\
		\delta g &= 0 &
		\delta \chi &= {\left(\nabla^s q\right)}_\sharp
	\end{align}
	where ${\left(\nabla^s q\right)}_\sharp\equiv \na^s_{e_\al}q\otimes e_\al\in\Gamma(S_g\otimes TM)$, given that \(\chi=0\) and the following holds:
	\begin{multline}\label{eq:DegeneratedSUSYCurvatureCondition}
		\int_M 6\left<\psi, R^N\left(\left<q, \psi\right>_{\gs}, \map_* e_\alpha\right) \gamma(e_\al)\psi\right> +\left<S\nabla R(\psi), \left<q, \psi\right>_{\gs}\right> \\
		- 4\left<SR(\psi), \gamma(\grad\map)q\right> \dd{vol}_g = 0
	\end{multline}
	The last condition is in particular fulfilled if the target manifold \(N\) is flat.
\end{prop}
\begin{remark}
	The degenerate super symmetry transformations of Proposition~\ref{prop:DegenerateSUSY} coincide up to the sign of \(\delta\psi\) and \(\delta\chi\) with the super symmetry transformations given in~\cite{kessler2016thesis} for the super conformal action functional.
	The sign difference is necessary to compensate for the use of the opposite Clifford algebra in this work.
\end{remark}
\begin{proof}
	The variation of \(\A(\map, \psi; g, \chi)\) is given by
	\begin{equation}
		\begin{split}
			\delta\A(\map, \psi; g, \chi) &= \int_M 2 g^{\alpha\beta}\map^*h\left(\nabla^{\map^*TN}_{e_\alpha}\left<q, \psi\right>, \map_*e_\beta\right) - \left< \gamma(e_\al) q\otimes \map_* e_\alpha, \D\psi\right> \\
				&\qquad - \left<\psi, \D\left(\gamma(e_\al) q\otimes \map_*e_\alpha\right) + R^N\left(\left<q, \psi\right>_{\gs}, \map_* e_\alpha\right) \gamma(e_\al)\psi\right> \\
				&\qquad + 2 \left< \gamma(e_\al)\gamma(e_\be) \nabla_{e_\alpha} q\otimes \map_*e_\beta , \psi\right> \\
				&\qquad - \frac16\left(\left<S\nabla R(\psi), \left<q, \psi\right>_{\gs}\right> - 4\left<SR(\psi), \gamma(\grad\map)q\right>\right) \dv
		\end{split}
	\end{equation}
	Here we have used
	\begin{align}
		\delta\left(\D\psi\right) &= \D\left(\delta\psi\right) + R^N\left(\delta\map, \map_* e_\alpha\right) \gamma(e_\al)\psi, \\
		\delta\left(R^N(\psi)\right) &= \left<S\nabla R(\psi), \delta\map\right> + 4 \left<SR(\psi), \delta\psi\right>,
	\end{align}
	compare~\cite[Section 4.1, (2) and (5)]{jost2016regularity}.
	If we now use~\eqref{eq:DegeneratedSUSYCurvatureCondition}, the variation of the action reduces to:
	\begin{align}
			\delta\A(\map&, \psi; g, \chi) = \int_M 2 g^{\alpha\beta}\map^*h\left(\nabla_{e_\alpha}\left<q, \psi\right>, \map_*e_\beta\right) - \left< \gamma\left(e_\alpha\right) q\otimes \map_* e_\alpha, \D\psi\right> \\
				&\hspace{6em} - \left<\psi, \D\left(\gamma\left(e_\alpha\right) q\otimes \map_*e_\alpha\right)\right> + 2 \left< \gamma\left(e_\alpha\right)\gamma\left(e_\beta\right) \nabla_{e_\alpha} q\otimes \map_*e_\beta , \psi\right>\dv_g \\
			&= \int_M 2g^{\alpha\beta} \left(\left<\nabla_{e_\alpha}q\otimes \map_*e_\beta, \psi\right> + \left<q\otimes \map_*e_\beta, \nabla_{e_\alpha} \psi\right>\right) + \left<\gamma\left(e_\beta\right)\gamma\left(e_\alpha\right) q\otimes \map_*e_\alpha, \nabla_{e_\beta}\psi\right> \\
				&\qquad - \left<\gamma\left(e_\beta\right)\left(\gamma\left(\nabla_{e_\beta}e_\alpha\right)q \otimes\map_*e_\alpha + \gamma\left(e_\alpha\right)\nabla_{e_\beta} q\otimes \map_*e_\alpha + \gamma\left(e_\alpha\right) q\otimes \nabla_{e_\beta}\map_*e_\alpha\right), \psi\right> \\
				&\qquad + 2 \left< \gamma\left(e_\alpha\right)\gamma\left(e_\beta\right) \nabla_{e_\alpha} q\otimes \map_*e_\beta , \psi\right>\dv_g \displaybreak[0]\\
			&= \int_M -\left<\gamma\left(e_\beta\right)\gamma\left(e_\alpha\right)\nabla_{e_\alpha}q\otimes \map_*e_\beta, \psi\right> - \left<\gamma\left(e_\beta\right)\gamma\left(e_\alpha\right) q\otimes \nabla_{e_\alpha}\map_*e_\beta, \psi\right> \\
				&\qquad - \left<\gamma\left(e_\beta\right)\gamma\left(e_\alpha\right) q\otimes \map_*e_\beta, \nabla_{e_\alpha}\psi\right> - \left<\gamma\left(e_\beta\right)\gamma\left(\nabla_{e_\beta}e_\alpha\right)q \otimes\map_*e_\beta, \psi\right> \\
				&\qquad + \left<\gamma\left(e_\beta\right)\gamma\left(e_\alpha\right) q\otimes \left(\nabla_{e_\alpha}\map_*e_\beta - \nabla_{e_\beta}\map_*e_\alpha\right), \psi\right> \dv_g \displaybreak[0]\\
			&= \int_M -e_\alpha\left<\gamma\left(e_\beta\right)\gamma\left(e_\alpha\right) q\otimes \map_*e_\beta, \psi\right> + \left<\left(\gamma\left(\nabla_{e_\alpha}e_\beta\right)\gamma\left(e_\alpha\right) + \gamma\left(e_\beta\right)\gamma\left(\nabla_{e_\alpha}e_\alpha\right)\right)q\otimes \map_*e_\beta, \psi\right> \\
				&\qquad - \left<\gamma\left(e_\beta\right)\gamma\left(\nabla_{e_\beta}e_\alpha\right)q \otimes\map_*e_\beta, \psi\right> + \left<\gamma\left(e_\beta\right)\gamma\left(e_\alpha\right) q\otimes \left(\nabla_{e_\alpha}\map_*e_\beta - \nabla_{e_\beta}\map_*e_\alpha\right), \psi\right> \dv_g \\
			&= 0.
	\end{align}
	Here in the last step we have used the torsion-freeness of the connection.

\end{proof}

\begin{thm}
	Suppose as in Proposition~\ref{prop:DegenerateSUSY} that \(\chi=0\) and condition~\eqref{eq:DegeneratedSUSYCurvatureCondition} is fulfilled.
	In addition we assume that \(\map\) and \(\psi\) fulfill the Euler--Lagrange equations.
	Then \(\diverg_g J=0\).
	In particular, the supercurrent \(J\) can be identified with a holomorphic section of \(TM\otimes_\C S\).
\end{thm}
\begin{proof}
	In the particular case that \(\map\) and \(\psi\) fulfill the Euler--Lagrange equations, we know that for all super symmetry parameters \(q\):
	\begin{equation}
		0 = \delta\A(\map, \psi; g, \chi)
		= \int_M \left<\delta \chi, J\right> \dv
		= \int_M \left<{(\nabla^s q)}_\sharp, J\right>\dv
		= -\int_M \left<q, \diverg_g J\right> \dv.
	\end{equation}

\end{proof}

In the remainder of this subsection we give an application of the degenerate super symmetry to the action functional of Dirac harmonic maps, with or without curvature term.
Recall that the functional of Dirac-harmonic maps with curvature term in~\cite{chen2008liouville} can be obtained from the functional~\eqref{eq:action functional} by setting the gravitino to zero.
It is a Corollary of Proposition~\ref{prop:DegenerateSUSY} that the functional of Dirac-harmonic maps with or without curvature terms also has a degenerate super symmetry.

\begin{cor}[Degenerate super symmetry of the functional of Dirac-harmonic maps with curvature term]%
\label{cor:DegenerateSUSYDiracHarmonicMapswithCurvature}
	Let \(q\in\Gamma(\SpinorBundleM_g)\) be a twistor spinor.
	The functional of Dirac-harmonic maps with curvature term is invariant under the following infinitesimal transformations
	\begin{align}
		\delta\map &= \left<q, \psi\right>_{\gs} &
		\delta\psi&= -\gamma(\grad\map)q
	\end{align}
	provided that the curvature condition~\eqref{eq:DegeneratedSUSYCurvatureCondition} holds.
\end{cor}
\begin{proof}
	The calculation proceeds as in the proof of Proposition~\ref{prop:DegenerateSUSY} with the additional condition that the term that describes the variation of the gravitino needs to be zero.
	The term of \(\delta\A(\map, \psi; g, \chi)\) that arises from the variation of the gravitino is
	\begin{equation}
		\int_M 2\left< \gamma\left(e_\alpha\right)\gamma\left(e_\beta\right) \nabla^s_{e_\alpha} q\otimes \map_*e_\beta , \psi\right>\dv_g
		= -\int_M 2\left< \left(2\nabla^s_{e_\beta} q + \gamma\left(e_\beta\right)\pd_g q\right)\otimes \map_*e_\beta , \psi\right>\dv_g.
	\end{equation}
	This term vanishes if \(q\) is a twistor spinor, i.e.\ for all vector field \(X\) it holds
	\begin{equation}
		\nabla^s_X q + \frac12\gamma(X)\pd_g q=0.
	\end{equation}

\end{proof}
\begin{remark}
	Notice that it is particular to the two-dimensional setup that the condition to be a twistor spinor or a holomorphic spinor are identical.
	Hence, Corollary~\ref{cor:DegenerateSUSYDiracHarmonicMapswithCurvature} does not come as a surprise, as also for super Riemann surfaces a holomorphic super symmetry leaves the \(\frac32\)-part of the gravitino invariant, see~\cite[Chapter 11.1]{kessler2016thesis}.
\end{remark}

Recall that the Dirac-harmonic map functional in~\cite{chen2006dirac} does not include the curvature term of the target manifold.
But note that a curvature term will arise when taking variations.
\begin{cor}[Degenerate super symmetry of the Dirac-harmonic map functional]%
\label{cor:DegenerateSUSYforDH}
	Let \(q\in\Gamma(\SpinorBundleM_g)\) be a twistor spinor.
	The functional of Dirac-harmonic maps is invariant under the following infinitesimal transformations
	\begin{align}
		\delta\map &= \left<q, \psi\right>_{\gs} &
		\delta\psi&= -\gamma(\grad\map)q
	\end{align}
	provided that the following curvature condition
	\begin{equation}\label{eq:DiracharmonicCurvatureCondition}
		\int_M\left\langle\psi, R\left(\langle q,\psi\rangle_{\gs},\phi_*e_\al\right)\gamma(e_\al)\psi\right\rangle_{\gs\otimes\map^*h}\dv_g =0.
	\end{equation}
	holds.
\end{cor}

In~\cite{chen2006dirac} two seemingly unrelated critical Dirac-harmonic maps have been constructed:
The trivial solution \((\map_0, 0)\), where \(\map_0\) is a harmonic map and the twistor spinor solution \((\map_0, \psi)\) where \(\psi=-\gamma(e_\alpha) q \otimes \map_*e_\alpha\) is constructed from a twistor spinor \(q\) and the harmonic map \(\map_0\).
We will now show that those two solutions are related via degenerate super symmetry.

Consider the family defined on the time interval \([0,1]\) given by
\begin{align}
	\map_t &= \map_0 &
	\psi_t &= -t \gamma(e_\alpha) q\otimes \map_*e_\alpha
\end{align}
for a twistor spinor \(q\in\Gamma(\SpinorBundleM_g)\).
The family \((\map_t, \psi_t)\) interpolates between the trivial and the twistor spinor solution via a family of degenerate super symmetries.
Indeed for every \(t=\tau\) we have
\begin{align}
	\left.\frac{\dd}{\dd t}\right|_{t=\tau} \map_t &= 0 = \left<q, \psi_\tau\right> &
	\left.\frac{\dd}{\dd t}\right|_{t=\tau} \psi_t &= -\gamma(e_\alpha) q \otimes \map_* e_\alpha = -\gamma(\grad\map) q.
\end{align}
Fortunately the condition~\eqref{eq:DiracharmonicCurvatureCondition} is fulfilled along this family, and hence we conclude that \((\map_t, \psi_t)\) is critical for all \(t\in[0,1]\).
Consequently, we have a critical family of degenerate super symmetries and the twistor spinor solution should be considered equivalent to the trivial solution.

We expect that more non-trivial critical Dirac-harmonic maps can be constructed with the help of super symmetry.
The difficulty lies in the construction of suitable families \((\map_t, \psi_t)\).

\subsection{Conclusion}
We have shown that the functional of Dirac-harmonic maps with gravitino is invariant under the rescaled conformal transformations, super Weyl transformations and diffeomorphisms.
The energy-momentum tensor and the supercurrent of the action have been calculated.
We found that the rescaled conformal invariance prescribes the trace of the energy-momentum tensor, whereas super Weyl invariance assures that \(J\) is a section of \(S_g\otimes_\C TM\) off shell.
The diffeomorphism invariance leads on shell to a coupled differential equation of divergence type involving the energy-momentum tensor and the supercurrent.

Furthermore, we have shown that in the case of vanishing gravitino and under certain conditions on the curvature of the target manifold a degenerate super symmetry leaves the action functional invariant infinitesimally.
The restricting conditions on the curvature and of the vanishing gravitino cannot be lifted in the setting of Dirac-harmonic maps but only in the world of super geometry.
The degenerate super symmetry leads on shell to a divergence-free supercurrent.
Hence, in that case the energy-momentum tensor can be identified with a holomorphic quadratic differential and the supercurrent with an holomorphic section of \(TM\otimes_\C\SpinorBundleM_g\).

Finally, we want to mention that the functional \(\A\) also possesses a \(U(1)\)-gauge symmetry.
A particular special case of this gauge symmetry is
\begin{equation}
	\A(\map,\psi;g,\chi)=\A(\map,-\psi; g,-\chi).
\end{equation}

\section{Appendix}\label{sect:formal divergence}

As we frequently use the divergence operators on various fields, we make a review in this section and explain their relations to Lie derivatives and the Cauchy--Riemann operators.

\subsection{}

Recall that on a Riemannian manifold $(M,g)$ with local coordinates $(x^\al)$, the divergence of a vector field $X=X^\al \frac{\p}{\p x^\al}$ is a function on $M$ defined by
\begin{equation}
	\diverg_g X=\frac{1}{\sqrt{\det g}}\frac{\p}{\p x^\al}\left(\sqrt{\det{g}}X^\al\right)
\end{equation}
where $\det{g}=\det{(g_{\al\be})}$. In terms of the Levi-Civita connection $\na^{LC}$ the divergence operator is expressed by
\begin{equation}
	\diverg_g (X)=\tr_g(\na X)\in C^\infty(M).
\end{equation}
In this way the divergence operator is the negative $L^2$-adjoint operator of the gradient operator: for any $f\in C^\infty(M)$ and any $X\in \Gamma(TM)$,
\begin{equation}
	\int_M \left\langle X,\grad(f)\right\rangle\dv_g=\int_M \left\langle -\diverg_g X,f\right\rangle\dv_g.
\end{equation}

The divergence operator on symmetric two-tensors is defined in a similar way. Let $k\in \Gamma(\Sym(T^*M\otimes T^*M))$, the divergence operator on $k$ is the one-form given by
\begin{equation}
	\diverg_g k= \tr_g(\na k)\coloneqq \sum_{\al=1}^{\dim M} \na_{e_\al} k(e_\al, \cdot)\in \Gamma(T^*M),
\end{equation}
where $(e_\al)$ is a local orthonormal frame.
In local coordinates $(x^\al)$, $k=k_{\al\be}\dd x^\al\otimes\dd x^\be$ with $k_{\al\be}=k_{\be\al}$ and let $K=K^\al_\be \dd x^\be\otimes \frac{\p}{\p x^\al}\in \End(TM)$ be the associated $(1,1)$-tensor,
\begin{equation}
	\diverg_g(k)=\left(\frac{1}{\sqrt{\det{g}}}\frac{\p}{\p x^\al}\left(\sqrt{\det{g}}K^\al_\be\right) - \frac{1}{2}g^{\al\eta}\frac{\p g_{\eta\gamma}}{\p x^\be} K^\gamma_\al\right)\dd x^\be.
\end{equation}
Then we claimed that, for any vector field $X\in\Gamma(TM)$,
\begin{equation}\label{eq:DivergenceLieAdjoint}
	\int_M \langle \Lie_X g, k\rangle\dv_g=-2\int_M \langle X,\diverg_g k\rangle\dv_g.
\end{equation}
Actually, using the local expression above,
\begin{equation}
	\begin{split}
		\int_M (\diverg_g k)(X)\dv_g
		&=\int_M \left(\frac{1}{\sqrt{\det{g}}}\frac{\p}{\p x^\al}\left(\sqrt{\det{g}}K^\al_\be\right) - \frac{1}{2}g^{\al\eta}\frac{\p g_{\eta\gamma}}{\p x^\be}K^\gamma_\al\right)X^\be \dv_g \\
		&=\int_M \frac{\p}{\p x^\al}\left(\sqrt{\det{g}}K^\al_\be X^\be\right)\dd x
			-\int_M K^\al_\be\frac{\p X^\be}{\p x^\al}+\frac{1}{2}g^{\al\eta}\frac{\p g_{\eta\gamma}}{\p x^\be}K^\gamma_\al X^\be \dv_g \\
		&=\int_M \diverg_g (K(X))\dv_g-\frac{1}{2}\int_M \langle k,\Lie_X g\rangle\dv_g.
	\end{split}
\end{equation}
Since $M$ is closed, the divergence theorem implies the first summand in the last line vanishes. Hence the claim is confirmed.

Recall that a symmetric, traceless and divergence-free two-tensor corresponds to a holomorphic quadratic differential on a Riemann surface, see e.g.~\cite[Section 2.4]{tromba2012teichmuller}.

\subsection{}
Now we consider the divergence operators on spinor fields.
Recall from~\cite{bourguignon1992spineurs} that on the spinor bundle $S_g$, the Lie derivative with respect to $X\in\Gamma(TM)$ on a spinor $\sigma\in \Gamma(S_g)$ is related to the spin connection $\na^s$ via
\begin{equation}
	\Lie_X^S\sigma =\na^s_X\sigma-\frac{1}{4}\gamma(\dd X^\flat)\sigma,
\end{equation}
where $X^\flat$ denotes the dual one-form of the vector field $X$ and the 2-form $\dd X^\flat$ acts via Clifford multiplication.
We would like to define a divergence operator on spinor fields such that a formula of the same type as~\eqref{eq:DivergenceLieAdjoint} holds.
This is achieved in the following lemma.
\begin{lemma}%
\label{lemma:SpinDivergence}
	Let $(M,g)$ be a Riemann surface, with almost complex structure $J_M\in \Aut(TM)$.
	For \(\rho\), \(\sigma\in\Gamma(S_g)\), define the divergence of $\rho$ with respect to $\sigma$ as
	\begin{equation}
		\diverg_\sigma(\rho)=\langle\na^s\sigma,\rho\rangle_\sharp+\frac{1}{4}J_M \grad\big(\langle\gamma(\omega)\sigma,\rho\rangle\big) \in \Gamma(TM),
	\end{equation}
	where $\omega$ stands for the volume element in the Clifford bundle.
	The following holds for all vector fields \(X\in\Gamma(TM)\)
	\begin{equation}
		\int_M \langle \Lie_X^S\sigma,\rho\rangle\dv_g=\int_M \langle X,\diverg_\sigma(\rho)\rangle\dv_g.
	\end{equation}

\end{lemma}

\begin{proof}
	Take local isothermal coordinates $(x^\al)$ and write $X=X^\al\frac{\p}{\p x^\al}$.
	Then $X^\flat=X^\al g_{\al\be}\dd x^\be\equiv X_\be \dd x^\be$.
	Since
	\begin{equation}
		\dd X^\flat=\left(\frac{\p X_2}{\p x^1}-\frac{\p X_1}{\p x^2}\right)\dd x^1\wedge \dd x^2,
	\end{equation}
	one sees that
	\begin{equation}
		\gamma(\dd X^\flat)=\left(\frac{\p X_2}{\p x^1}-\frac{\p X_1}{\p x^2}\right)\frac{1}{\sqrt{\det{g}}}\gamma(\omega).
	\end{equation}
	Thus,
	\begin{equation}
		\begin{split}
		\int_M \langle\gamma(\dd X^\flat)\sigma, \rho\rangle\dv_g
		=&\int_M \left\langle\left(\frac{\p X_2}{\p x^1}-\frac{\p X_1}{\p x^2}\right)\frac{1}{\sqrt{\det{g}}}\gamma(\omega)\sigma,\rho\right\rangle \sqrt{\det{g}}\dd x \\
		=&\int_M \frac{\p X_2}{\p x^1}\langle\gamma(\omega)\sigma,\rho\rangle-\frac{\p X_1}{\p x^2}\langle\gamma(\omega)\sigma, \rho\rangle \dd x \\
		=&\int_M -X_2\frac{\p}{\p x^1}\left(\langle\gamma(\omega)\sigma,\rho\rangle\right)
			+X_1\frac{\p}{\p x^2}\left(\langle\gamma(\omega)\sigma,\rho\rangle\right) \dd x \\
		=&\int_M \big\langle *X^\flat, \dd\langle\gamma(\omega)\sigma,\rho\rangle\big\rangle \dv_g \\
		=&\int_M \big\langle J_M X,\grad(\langle\gamma(\omega)\sigma,\rho\rangle)\big\rangle \dv_g \\
		=&\int_M \big\langle X,-J_M \grad(\langle\gamma(\omega)\sigma,\rho\rangle)\big\rangle \dv_g.
		\end{split}
	\end{equation}
	Therefore,
	\begin{equation}
		\begin{split}
			\int_M \langle\Lie^S_X\sigma,\rho\rangle\dv_g
			&=\int_M \langle \na^s_X\sigma,\rho\rangle-\frac{1}{4}\langle\gamma(\dd X^\flat)\sigma,\rho\rangle \dv_g \\
			&=\int_M \big\langle X, \langle\na^s\sigma,\rho\rangle_{\sharp}
				+\frac{1}{4}J_M \grad(\langle\gamma(\omega)\sigma,\rho\rangle)\big\rangle \dv_g.
		\end{split}
	\end{equation}
\end{proof}
\begin{remark}
	In contrast to the divergence operators defined for the vector fields and for the symmetric 2-tensors, this divergence operator $\diverg_\sigma(\rho)$ doesn't involve derivatives of $\rho$, but only derivatives of $\sigma$.
	In this sense, it is only a formal ``divergence'' operator.
\end{remark}

\subsection{}
Next we also need to consider the divergence operators defined on the tensor product bundle $S_g\otimes TM$.
Let $\varphi=\varphi^\al\otimes e_\al\in\Gamma(S_g\otimes TM)$ be a gravitino. In a local orthonormal frame $(e_\al)$, the $g$-divergence operator is
\begin{equation}
	\diverg_g(\varphi)\coloneqq \tr_g\left(\widehat{\na} \varphi\right)=\sum_{\al}\left<\widehat{\na}_{e_\al}\varphi, e_\al\right>\in\Gamma(S_g),
\end{equation}
where we use $\widehat{\na}$ to denote the connection of $S_g\otimes TM$.
For any spinor field $q\in\Gamma(S_g)$, using integration by parts,
\begin{equation}
	\begin{split}
		\int_M \langle q,\diverg_g \varphi\rangle\dv_g
		&=\int_M \left< q, \langle \widehat{\na}_{e_\al}\varphi,e_\al\rangle \right>\dv_g \\
		&=\int_M \left< q\otimes e_\al, \widehat{\na}_{e_\al}\varphi\right> \dv_g \\
		&=-\int_M \left<\na^s_{e_\al}q\otimes e_\al, \varphi\right>\dv_g\\
		&=-\int_M \left< {(\na^s q)}_\sharp, \varphi \right>\dv_g
	\end{split}
\end{equation}
Note that $P\varphi=0$ implies that $\varphi$ is a smooth section of the complex vector bundle $S_g\otimes_\C TM$. If in addition $\diverg_g \varphi=0$, then $\varphi$ is then holomorphic, which is to say, $\varphi^\vee\equiv \varphi_\al\otimes e^\al\in\Gamma(S_g^\vee\otimes T^*M) $ is a holomorphic section.

We also need the $\chi$-divergence of $\varphi$ for a gravitino field $\chi\in\Gamma(S_g\otimes TM)$.
It is formally defined to make the following identity hold:
\begin{equation}
	\int_M \langle \Lie^{S_g\otimes TM}_X \chi, \varphi\rangle\dv_g=\int_M \langle X, \diverg_\chi \varphi\rangle\dv_g
\end{equation}
This can be assured using the Riesz representation theorem.

\end{document}